\documentclass[10pt]{amsart}

\usepackage{amsmath,amssymb,amsfonts}
\usepackage{esint, cancel}
\usepackage{amsthm}
\usepackage{mathrsfs}
\usepackage{geometry}
\usepackage{verbatim}
\usepackage{graphicx}
\usepackage{color}
\usepackage[toc,page]{appendix}

\usepackage[colorlinks=true,urlcolor=blue, citecolor=red,linkcolor=blue,
linktocpage,pdfpagelabels, bookmarksnumbered,bookmarksopen]{hyperref}
\usepackage[hyperpageref]{backref}

\newtheorem{theorem}{Theorem}[section]
\newtheorem{lemma}[theorem]{Lemma}
\newtheorem{proposition}[theorem]{Proposition}
\newtheorem{corollary}[theorem]{Corollary}

\theoremstyle{definition}
\newtheorem{remark}[theorem]{Remark}
\newtheorem{example}[theorem]{Example}
\newtheorem{definition}[theorem]{Definition}
\newtheorem{open}{Open problem}
\newtheorem*{ack}{Acknowledgments}

\numberwithin{equation}{section}

\title{Variations on the capacitary inradius}
\author[Bozzola]{Francesco Bozzola}
\address[F.\ Bozzola]{DIME Dipartimento di ingegneria meccanica, energetica, gestionale e dei trasporti 
	\newline\indent
	Universit\`a di Genova
	\newline\indent
via alla Opera Pia 15, 16145 Genova, Italy}
\email{francesco.bozzola@edu.unige.it}

\author[Brasco]{Lorenzo Brasco}
\address[L.\ Brasco]{Dipartimento di Matematica e Informatica
	\newline\indent
	Universit\`a degli Studi di Ferrara
	\newline\indent
Via Machiavelli 35, 44121 Ferrara, Italy}
\email{lorenzo.brasco@unife.it}

\dedicatory{Remembering Roberto Franceschi, gone too soon}

\date{\today}

\begin{document}

\begin{abstract}
We discuss some properties of the {\it capacitary inradius} for an open set. This is an extension of the classical concept of {\it inradius} (i.e. the radius of a largest inscribed ball), which takes into account capacitary effects. Its introduction dates back to the pioneering works of Vladimir Maz'ya. We present some variants of this object and their mutual relations, as well as their connections with Poincar\'e inequalities. We also show that, under a mild regularity assumption on the boundary of the sets, the capacitary inradius is equivalent to the classical inradius. This comes with an explicit estimate and it permits to get a Buser--type inequality for a large class of open sets, whose boundaries may have power-like cusps of arbitrary order.
Finally, we present a couple of open problems.
\end{abstract}

	\maketitle

\begin{center}
\begin{minipage}{10cm}
\small
\tableofcontents
\end{minipage}
\end{center}

\section{Introduction}

\subsection{Motivation}
In this note, we pursue our study on the {\it capacitary inradius} of open subsets of $\mathbb{R}^N$, that we started in our recent paper \cite{BozBra2}. In order to gently introduce the reader to the subject, it is certainly useful to start by recalling the definition of {\it inradius} of an open set $\Omega\subseteq\mathbb{R}^N$: this is the geometric quantity given by
\[
r_\Omega=\sup\Big\{r>0\, :\, \text{there exists a ball}\ B_r(x_0)\subseteq\Omega\Big\}.
\]
It can be seen as a simple (and quite rough) measure of ``fatness'' of an open set. In connection with functional inequalities, its importance is due to the following fact: the condition $r_\Omega<+\infty$ is {\it necessary} in order to infer the existence of a constant $c_\Omega>0$ such that
\[
c_\Omega\,\int_\Omega |\varphi|^p\,dx\le \int_\Omega |\nabla\varphi|^p\,dx,\qquad \text{for every}\ \varphi\in C^\infty_0(\Omega),
\]
where $1\le p<\infty$ (see for example \cite[Proposition 2.1]{Sou}).
Thus, the finiteness of the inradius is a necessary condition for the validity of the Poincar\'e inequality. If we define the sharp constant in the previous inequality, i.e.
\[
\lambda_p(\Omega):=\inf_{\varphi\in C^\infty_0(\Omega)} \left\{\int_\Omega |\nabla \varphi|^p\,dx\, :\, \|\varphi\|_{L^p(\Omega)}=1\right\},
\]
the previous condition can be expressed in a quantitative way through the following well-known sharp estimate
\[
\lambda_p(\Omega)\le \frac{\lambda_p(B_1)}{r_\Omega^p}.
\]
We briefly recall that in general it is not always possible to reverse the previous inequality, unless some further restrictions on the open sets are imposed. We refer to the introductions of our previous papers \cite{BozBra, BozBra2} for a list of results in this direction (see also \cite{Boz} and the references therein included). 
The crucial obstruction to the validity of a lower bound of the type
\begin{equation} \label{gallina-vecchia}
	\frac{C_{N,p}}{r_\Omega^p} \le \lambda_p(\Omega),
\end{equation}
is a ``removability issue'': roughly spealing, while $r_\Omega$ can be altered by the removal of a single point, the Poincar\'e constant $\lambda_p(\Omega)$ is only affected by the removal of ``sufficiently large'' sets. This said, it is quite easy to produce counterexamples to the validity of \eqref{gallina-vecchia}.
The previous largeness condition depends on the exponent $p$ and is expressed in terms of $p-$capacity. In this paper, we will mainly work with the so-called {\it relative $p-$capacity}, defined by
\[
\mathrm{cap}_p(\Sigma;E)=\inf_{\varphi\in C^\infty_0(E)}\left\{\int_E |\nabla \varphi|^p\,dx\, :\, \varphi \ge 1 \text{ on } \Sigma\right\},
\]
for every open bounded set $E \subseteq \mathbb{R}^N$ (typically, a ball) and every compact set $\Sigma\Subset E$.
\vskip.2cm\noindent
Under this premise, it is natural to introduce a variant of $r_\Omega$ which disregards sets having zero $p-$capacity, i.e. we wish to consider a relaxed notion of inradius which shares with $\lambda_{p}(\Omega)$ the same removable sets. This is the idea for introducing the {\it capacitary inradius} of an open set $\Omega$, whose precise definition is as follows:
 for every $0<\gamma<1$ and for every $1 \le p < \infty$, this is defined by 
\begin{equation}
\label{cap_inradius}
R_{p,\gamma}(\Omega):=\sup\Big\{r>0\, :\, \exists x_0\in\mathbb{R}^N\ \text{such that}\ \overline{B_r(x_0)}\setminus\Omega\ \text{is $(p,\gamma)-$negligible}\Big\}.
\end{equation}
Here $(p,\gamma)-$negligible (in the sense of Molchanov) means that $\overline{B_r(x_0)}\setminus\Omega$ occupies a portion of a reference concentric ball, say $B_{2r}(x_0)$, which is at most $\gamma$ in the sense of $p-$capacity. More precisely, we have 
\[
\mathrm{cap}_p\left(\overline{B_r(x_0)}\setminus\Omega;B_{2r}(x_0)\right)\le \gamma\,\mathrm{cap}_p\left(\overline{B_r(x_0)};B_{2r}(x_0)\right).
\]
The quantity \eqref{cap_inradius} has been explicitly introduced in our recent paper \cite{BozBra2}, by drawing inspiration both from the
{\it interior capacitary radius} used by Maz'ya and Shubin in \cite{MS} 
and from the {\it inner cubic diameter} (see \cite[Definition 14.2.2]{Maz}), which has been extensively used by Maz'ya. 
We postpone to the following section a more detailed discussion about the comparison between these capacitary variants of the inradius. 
\begin{remark}
The extremal cases $\gamma=1$ and $\gamma=0$ deserve a comment: in the first case, every set would be $(p,\gamma)-$negligible and thus the corresponding capacitary inradius would be $+\infty$ for every open set. The case $\gamma=0$ is more interesting: in this case, the quantity $R_{p,0}(\Omega)$ would give the radius of the largest ball contained in $\Omega$, {\it up to a set of zero $p-$capacity}.
However, this capacitary inradius would not be strong enough to permit the lower bound
\[
\left(\frac{1}{R_{p,0}(\Omega)}\right)^p\lesssim \lambda_p(\Omega),
\]
see \cite[Example A.1]{BozBra2} for a counter-example.
\end{remark}
On the contrary, when $0<\gamma<1$ one can prove that $R_{p,\gamma}(\Omega)$ has the desired property: it is actually equivalent to $\lambda_p(\Omega)$. More precisely, in \cite[Main Theorem]{BozBra2} we proved the following two-sided estimate
\begin{equation}
\label{capin}
\sigma_{N,p}\,\gamma\,\left(\frac{1}{R_{p,\gamma}(\Omega)}\right)^p\le \lambda_p(\Omega)\le C_{N,p,\gamma}\,\left(\frac{1}{R_{p,\gamma}(\Omega)}\right)^p,
\end{equation}
which holds {\it for every} $0<\gamma<1$ and within the range $1 \le p \le N$. Actually, a related result holds for the sharp Poincar\'e-Sobolev constants (see \cite[Theorem 6.1]{BozBra2})
\[
\lambda_{p, q}(\Omega)=\inf_{\varphi\in C^\infty_0(\Omega)} \left\{\int_\Omega |\nabla \varphi|^p\,dx\, :\, \|\varphi\|_{L^q(\Omega)}=1\right\},
\]
each time $q>p$ is a {\it subcritical} (in the sense of Sobolev embeddings) exponent.  
\begin{remark}[The case $p>N$]
The regime $p>N$ is less interesting. In this case, points have positive $p-$capacity and thus they are not removable sets. Accordingly, one can prove that \eqref{gallina-vecchia} holds for every open sets\footnote{This result has been proved and re-proved various times, with slightly different proofs and different estimates of the constant $C_{N,p}$.
However, it seems correct to attribute it to Maz'ya, see for example \cite[Theorem 11.4.1]{Maz85}.}.
We also mention that in \cite[Proposition 7.3]{BozBra2} it is shown that $R_{p,\gamma}(\Omega)=r_\Omega$, when $p>N$ and $0<\gamma\le \gamma_0$, for an explicit (and optimal) parameter $\gamma_0$. 
For these reasons, in this paper we will confine our discussion to the case $p\le N$.
\end{remark}
\subsection{Description of the results}

In the first part of the paper, we want to consider some variants of $R_{p,\gamma}$ and discuss their mutual relations. We will focus especially on two different kinds of variants:
\vskip.2cm
\begin{itemize}
\item[(A)] changing the {\it metric}, i.e. replacing balls in the definition of $R_{p,\gamma}$ with a family of rescaled copies of a more general fixed ``shape'';
\vskip.2cm
\item[(B)] changing the {\it capacity}, i.e. replacing the relative $p-$capacity with other notions of $p-$capacity. 
\end{itemize}
\vskip.2cm
The first question may look a bit academic, but actually it is not: indeed, the first appearing of the capacitary inradius can be traced back in some works by Maz'ya from the '70s of the 20th century, as summarized in his books \cite{Maz,Maz85}. In these works, the capacitary inradius is defined by using {\it cubes} in place of {\it balls}: it is precisely the {\it inner cubic diameter} mentioned above (see \cite[Definition 10.2.2]{Maz85} and \cite[Definition 14.2.2]{Maz}). It is crucially exploited in order to derive necessary and sufficient conditions for the validity of Poincar\'e and Poincar\'e-Sobolev inequalities on open subsets of $\mathbb{R}^N$.
The use of cubes has the advantage that they can tile the space, differently from balls; on the other hand, by using balls it is simpler to prove a two-sided estimate like \eqref{capin} valid {\it for every} $0<\gamma<1$, as shown in \cite{MS} (for $p=2<N$)
and in our paper \cite{BozBra2} (general case). It is thus natural to inquire to which extent the shape of the chosen ``ball'' can affect this type of estimates. It should be also remarked that this kind of generalization has been considered already in \cite[Chapter 18]{Maz} (for the case $p=2$). 
\par
As for point (B), we remark that we do not perversely consider any possible variant one could imagine: rather, our interest is to consider a couple of alternative notions already existing in the literature. The first one is what we call the {\it Maz'ya-Shubin capacitary inradius}, used in \cite{MS}, a paper which very much inspired \cite{BozBra2}. By referring to Definition \ref{def:ms-capin} for its precise form, we point out that the main difference with $R_{p,\gamma}$ is the use of the following {\it absolute} $p-$capacity
\[
\mathrm{cap}_p(\Sigma)=\inf_{\varphi\in C^\infty_0(\mathbb{R}^N)} \left\{\int_{\mathbb{R}^N} |\nabla \varphi|^p\,dx\, :\, \varphi\ge 1\ \text{on}\ \Sigma\right\}.
\] 
As already explained in the Introduction of \cite{BozBra2}, the main drawback of this notion is that it does not
permit to consider the limit case $p=N$.
\par
The second variant is the {\it Gallagher capacitary inradius}, recently considered in \cite{Ga1, Ga2}. Apart for the use of yet another notion of $p-$capacity (which permits to include the case $p=N$, as well), the main difference is that this inradius {\it is not} defined in terms of a negligibility condition {\it \`a la} Molchanov. Rather, it can be regarded as a sort of ``endpoint'', i.e. it coincides with the limit of $R_{p,\gamma}$ as the negligibility parameter $\gamma$ goes to $0$ (see Proposition \ref{prop:GvsUs}). This has already been observed by Gallagher in \cite{Ga2}. Thus, even if this inradius can be used to characterize the validity of Poincar\'e inequality, it {\it can not} be directly used to provide two-sided estimates like \eqref{capin} (more precisely, the lower bound fails, see Example \ref{exa:0fesso}).
\vskip.2cm\noindent
The second part of the paper is focused on answering the following natural question: {\it under which conditions on the open set, the capacitary inradius $R_{p,\gamma}(\Omega)$ can be compared with the classical inradius $r_\Omega$?}
In Section \ref{sec:5}, we will show how a simple {\it measure density condition} on the open sets imply that 
\[
r_\Omega\le R_{p,\gamma}(\Omega)\lesssim r_\Omega,\qquad \text{for every}\ 0<\gamma<1.
\]
Thus, the capacitary inradius $R_{p,\gamma}(\Omega)$ turns out to be equivalent to the usual inradius for these sets. This is the content of Theorem \ref{prop:meas-density-ball} below. Our measure density condition amounts to require the following
\[
\theta^*_{\Omega,r_0}(t):=\inf\left\{\left(\frac{r_0}{r}\right)^t\,\dfrac{|B_r(x)\setminus \Omega|}{|B_r(x)|}\, : x\in \partial \Omega,\ 0 < r \le r_0\right\}>0,
\]
for some $r_0>0$ and $t\ge 0$.
Such a condition is certainly not the sharpest possible assumption on the open sets, in order to have such an equivalence. However, we think that the result is quite interesting for two reasons: the condition is sufficiently general to encompass a large class of open sets (we admit power-like cusps of arbitrary order for their boundaries); at the same time, it is easy to check. We refer to Lemma \ref{lm:funnel} for a large class of open sets which satisfy it.
Moreover, the equivalence between $R_{p,\gamma}(\Omega)$ and $r_\Omega$ comes with an explicit estimate. 
\par
When combined with the lower bound in \eqref{capin}, this estimate in turn permits to infer the validity of \eqref{gallina-vecchia}
for this class of sets, with a precise control on the constant involved. At the same price, we can get the same result also for the Poincar\'e-Sobolev constants $\lambda_{p,q}$, see Corollary \ref{coro:kroffo}. 
As a remarkable consequence, we can obtain a {\it Buser-type inequality} for open sets satisfying the previous measure density condition. More precisely, we get 
\begin{equation}
\label{busa}
\lambda(\Omega)\le C\,\Big(h(\Omega)\Big)^2,
\end{equation}
where $C>0$ is constant depending only on $N$ and the measure density index $\theta_{\Omega,r_0}^*(t)$ (in an explicit way) and:
\begin{itemize}
\item $\lambda(\Omega)$ is the bottom of the spectrum of the Dirichlet-Laplacian on $\Omega$, that is
\[
\lambda(\Omega):=\inf_{\varphi\in C^\infty_0(\Omega)} \left\{\int_\Omega |\nabla \varphi|^2\,dx\, :\, \|\varphi\|_{L^2(\Omega)}=1\right\};
\]
\vskip.2cm
\item $h(\Omega)$ is the {\it Cheeger constant of $\Omega$}, defined by
\[
h(\Omega)=\inf\left\{\frac{\mathcal{H}^{N-1}(\partial E)}{|E|}\, :\, E\Subset \Omega\ \text{open set with smooth boundary}\right\}.
\]
\end{itemize}
We refer to Corollary \ref{coro:buser} for the precise statement. 
\par
We conclude by recalling that for general open sets it is not possible to have an estimate like \eqref{busa}: we refer to \cite[Chaper 4, Section 3]{Maz} for a counter-example and to \cite{BozBra, Bra, Par} for some positive results, under suitable topological/geometric assumptions on the open sets. Inequality \eqref{busa} is sometimes also called {\it reverse Cheeger's inequality} (see for example \cite{Bu,Bu2, Le2} and \cite{Le} for this result in the context of Riemannian manifolds). The motivation for this terminology is easily understood, once we recall the following {\it Cheeger inequality}
\begin{equation}
\label{ciga}
\left(\frac{h(\Omega)}{2}\right)^2\le \lambda(\Omega),
\end{equation}
see \cite[equation (4.2.5)]{Maz}, which holds for {\it every} open set $\Omega\subseteq\mathbb{R}^N$.

\subsection{Plan of the paper}
In Section \ref{sec:2}, we settle the notation and state some basic facts which will be repeatedly used throughout the paper. In Section \ref{sec:3}, we compare our capacitary inradius with the variant obtained by considering a more general class of test ``shapes", other than balls. An interesting open problem is presented, as well. Section \ref{sec:4} 
is still devoted to comparison estimates with other notions of capacitary inradius, this time obtained by considering different capacities. These variants includes the one considered by Maz'ya and Shubin in \cite{MS} and the one considered by Gallagher in \cite{Ga1, Ga2}. This section is complemented with examples and remarks, aimed at clarifying some substantial differences between $R_{p, \gamma}$ and the inradius considered in \cite{Ga1, Ga2}. In the final section, i.e. Section \ref{sec:5}, we state and prove our main result, about the comparison between the capacitary inradius and the classical one, under a measure density condition (see Theorem \ref{prop:meas-density-ball}). We also briefly discuss some consequences which can be drawn from it. 
\begin{ack}
We are grateful to Vladimir Bobkov for drawing our attention to the papers \cite{Ga1, Ga2}. 
F.\,B. is a member of the {\it Gruppo Nazionale per l'Analisi Matematica, la Probabilit\`a
e le loro Applicazioni} (GNAMPA) of the Istituto Nazionale di Alta Matematica (INdAM) and partially supported by the ``INdAM - GNAMPA Project Ottimizzazione Spettrale, Geometrica e Funzionale", CUP E5324001950001. This paper has been finalized during the {\sc XXXIV Convegno Nazionale di Calcolo delle Variazioni}, held in Riccione in February 2025.
\end{ack}

\section{Preliminaries}
\label{sec:2}
Unless differently stated, in this paper we will always take the dimension to be $N\ge 2$.
For a point $x_0\in\mathbb{R}^N$ and a positive real number $R$, we will indicate by $B_R(x_0)$ the $N-$dimensional open ball, centered at $x_0$ and with radius $R$, i.e.
\[
B_R(x_0)=\Big\{x\in\mathbb{R}^N\, :\, |x-x_0|<R\Big\}.
\]
When the center $x_0$ coincides with the origin, we will simply write $B_R$. We indicate by $\omega_N$ the volume of the $N-$dimensional ball having radius $1$. We also denote by $Q_R(x_0)$ the $N-$dimensional open hypercube centered at $x_0$, given by 
\[
Q_R(x_0)=\prod_{i=1}^N (x_0^i-R,x_0^i+R),\qquad \text{where}\ x_0=(x_0^1,\dots,x_0^N).
\]
We recall the definition of $p-$capacity we wish to work with.
\begin{definition}
\label{defi:capacity}
Let $1\le p<\infty$, for every $E\subseteq\mathbb{R}^N$ open set and every $\Sigma\subseteq E$ compact set, we define the {\it $p-$capacity of $\Sigma$ relative to $E$} as follows
\[
\mathrm{cap}_p(\Sigma;E)=\inf_{\varphi\in C^\infty_0(E)}\left\{\int_E |\nabla \varphi|^p\,dx\, :\, \varphi \ge 1 \text{ on } \Sigma\right\}.
\]
By using standard approximation methods, it is not difficult to see that this infimum is unchanged, if we replace $C^\infty_0(E)$ by the space of Lipschitz functions, compactly supported in $E$.
\end{definition}
We refer to \cite{Maz} for a comprehensive study on the properties of the various notions of $p-$capacity used in this paper. We also mention the paper \cite{Ho} for a succint presentation of the subject, in the case of the relative $p-$capacity introduced above.
\vskip.2cm 
From its definition, it is easy to see that we have the following monotonicity relations
\begin{equation}
\label{monotone}
\mathrm{cap}_p(\Sigma_0;E)\le \mathrm{cap}_p(\Sigma_1;E),\qquad \text{if}\ \Sigma_0\subseteq\Sigma_1\Subset E,
\end{equation}
and
\[
\mathrm{cap}_p(\Sigma;E_1)\le \mathrm{cap}_p(\Sigma;E_0),\qquad \text{if}\ \Sigma\Subset E_0\subseteq E_1,
\]
that will be used repeatedly. 
\par
We also recall the following explicit formula for the $p-$capacity of concentric balls: they can be found for example in \cite[page 148]{Maz}. We limit ourselves to the case $1\le p\le N$, which will be the most interesting one:
\begin{equation}
\label{eqn:cap-ball0}
\mathrm{cap}_1\left(\overline{B_r};B_{R}\right)=N\,\omega_N\,r^{N-1},
\end{equation}
\begin{equation}
\label{eqn:cap-ball}
\mathrm{cap}_p\left(\overline{B_r};B_{R}\right)=N\,\omega_N\, \left(\frac{N-p}{p-1}\right)^{p-1}\,\frac{r^{N-p}}{\left(1-\left(\dfrac{r}{R}\right)^\frac{N-p}{p-1}\right)^{p-1}}, \qquad \text{ if } 1<p<N,
\end{equation}
and 
\begin{equation}
\label{eqn:cap-ballN}
\mathrm{cap}_N\left(\overline{B_r};B_{R}\right)=N\, \omega_N\, \left(\log \left(\frac{R}{r}\right)\right)^{1-N}.
\end{equation}
The next simple result is well-known (see for example \cite[Corollary 2.3.4]{Maz}), we recall its proof for the reader's convenience. It will be crucially exploited in the sequel.
\begin{lemma}
\label{lemma:2}
Let $E\subseteq\mathbb{R}^N$ be an open set and let $\Sigma\subseteq E$ be a compact set. For every $1\le p<\infty$ we have 
\[
|\Sigma|\,\lambda_p(E)\le \mathrm{cap}_p(\Sigma;E).
\]
\end{lemma}
\begin{proof}
We can suppose that $\lambda_p(E)>0$, otherwise there is nothing to prove. Let $\varphi\in C^\infty_0(E)$ be such that $\varphi\ge 1$ on $\Sigma$. Then, we have 
\[
\lambda_p(E)\,|\Sigma|\le \lambda_p(E)\,\int_\Sigma |\varphi|^p\,dx\le \lambda_p(E)\, \int_E |\varphi|^p\,dx\le \int_E |\nabla \varphi|^p\,dx.
\]
Taking the infimum over $\varphi$, we get the claim.
\end{proof}
The following property immediately follows from the explicit expressions \eqref{eqn:cap-ball0}, \eqref{eqn:cap-ball} and \eqref{eqn:cap-ballN}.  
\begin{lemma}
\label{lemma:3}
Let $1\le p\le N$. For every $\alpha>1$, the function
\[
f(r)=\mathrm{cap}_p\left(\overline{B_r(x_0)};B_{\alpha\,r}(x_0)\right),\qquad \text{for}\ r>0.
\]
is non-decreasing.
\end{lemma}
\section{Different metrics}
\label{sec:3}
Following \cite[Definition 18.1]{Maz}, we want to consider a variant of our capacitary inradius $R_{p,\gamma}$, where we replace balls with more general ``shapes''. We first need to define the kind of ``shapes'' we wish to consider.
\par
We say that an open bounded set $\mathcal{K} \subseteq \mathbb{R}^N$ is a {\it standard body} if it is starshaped with respect to the ball $B_1$.
Such a set $\mathcal{K}$ can be described as follows
\[
	\mathcal{K} = \{x = r\, \omega \in \mathbb{R}^N\,:\, \omega \in \mathbb{S}^{N-1},\, 0 \le r < r(\omega)\},
\]
for a Lipschitz function $r: \mathbb{S}^{N-1} \to (0, \infty)$, the so--called {\it radial function of $\mathcal{K}$}, given by 
\[
r(\omega) := \sup\{r \ge 0\ : \ r\,\omega \in \mathcal{K}\}, \qquad \mbox{for every}\ \omega \in \mathbb{S}^{N-1},
\]
see \cite[Lemma 1.1.8]{Maz}.
Moreover, we have
\[
	r(\omega)\,\omega \in \partial \mathcal{K} \qquad \mbox{ and } \qquad 1\le r(\omega) \le R_\mathcal{K}, \qquad \mbox{ for every } \omega \in \mathbb{S}^{N-1},
\]
where we set 
\[
R_\mathcal{K} := \displaystyle \max_{x \in \partial \mathcal{K}} |x| = \max_{\omega \in \mathbb{S}^{N-1}} r(\omega).
\]
Observe that we used that $B_1 \subseteq \mathcal{K}$, by our assumption. 
\par
For every $r > 0$ and for every $x_0 \in \mathbb{R}^N$, we denote by $\mathcal{K}_r(x_0)$ the open set obtained from $\mathcal{K}$ by combining an homothety (centered at $0$) with coefficient $r$ and a translation by the vector $x_0$, that is
\[
\mathcal{K}_r(x_0):= \left\{x \in \mathbb{R}^N\,:\,\frac{x-x_0}{r} \in \mathcal{K}\right\}.
\]
When the ``center'' $x_0$ coincides with the origin, we will simply write $\mathcal{K}_r$ in place of $\mathcal{K}_{r}(x_0)$. For $r=1$ we get back the original shape, thus we will write $\mathcal{K}$ in place of $\mathcal{K}_1$.
\par
This family of sets will replace the balls $B_r(x_0)$ in the capacitary inradius we are going to define. 
Indeed, we have the following definition, which is essentially contained in \cite[Chapter 18]{Maz}.
\begin{definition}
Let $1\le p<\infty$ and $0 < \gamma < 1$. A compact set $\Sigma\subseteq \overline{\mathcal{K}_r(x_0)}$ is said to be {\it $(p,\gamma)-$negligible relative to $\mathcal{K}$} if 
\[
\mathrm{cap}_p(\Sigma;\mathcal{K}_{2r}(x_0))\le \gamma\,\mathrm{cap}_p\left(\overline{\mathcal{K}_r(x_0)};\mathcal{K}_{2r}(x_0)\right).
\]
For an open set $\Omega \subseteq \mathbb{R}^N$ we define its {\it capacitary inradius relative to $\mathcal{K}$} as 
\begin{equation*}
	R_{p, \gamma}(\Omega; \mathcal{K}):= \sup\Big\{r>0\, :\, \exists x_0\in\mathbb{R}^N\ \text{such that}\ \overline{\mathcal{K}_r(x_0)}\setminus\Omega\ \text{is $(p,\gamma)-$negligible relative to}\ \mathcal{K}\Big\}. 
\end{equation*}
\end{definition}
A couple of comments are in order, concerning the previous definition.
\begin{remark}
A particularly interesting instance is when $\mathcal{K}$ coincides with the unit ball of a norm on $\mathbb{R}^N$.
For example, by choosing $\mathcal{K}=B_1$ the Euclidean ball, one gets back our initial definition. On the other hand, by choosing the hypercube
\[
\mathcal{K}=Q_1(0)=(-1,1)^N,
\]
i.e. the unit ball of the norm $\|\cdot\|_{\ell^\infty}$, 
the quantity $R_{p, \gamma}(\Omega; \mathcal{K})$ coincides with the so-called {\it inner cubic diameter}, extensively used by Maz'ya in his works, as recalled in the Introduction.
\end{remark}
\begin{remark}
We point out that the requirement on the standard body $\mathcal{K}$ to be starshaped with respect {\it to a ball} is needed in order to guarantee that $\mathcal{K}_r(x_0)$ is compactly contained in $\mathcal{K}_{2r}(x_0)$. In this way, the quantity $\mathrm{cap}_p(\overline{\mathcal{K}_r(x_0)};\mathcal{K}_{2r}(x_0))$ is well-defined.
\end{remark}
\begin{lemma} \label{lm:cap_palleVSbody}
Let $1\le p<\infty$, for every compact set $\Sigma\subseteq B_\varrho(x_0)$ we have
	\begin{equation}
	\label{francesco}
		\,\mathrm{cap}_p(\Sigma; B_\varrho(x_0)) \leq\left(\dfrac{\varrho}{\mathrm{dist}(\Sigma,\partial B_\varrho(x_0))}\,\frac{1}{\lambda_{p}(\mathcal{K})^{\frac{1}{p}}}+1\right)^{p}\, \mathrm{cap}_p(\Sigma; \mathcal{K}_\varrho(x_0)). 
	\end{equation}
\end{lemma}
\begin{proof}
We first observe that if $\Sigma\subseteq B_\varrho(x_0)$, then we have $\Sigma\subseteq \mathcal{K}_{\varrho}(x_0)$, as well. This simply follows from the fact that
\[
B_\varrho(x_0)\subseteq \mathcal{K}_\varrho(x_0),
\]
thanks to the fact that the standard body $\mathcal{K}$ is starshaped with respect to $B_1$. 
\par
We set for simplicity $d=\mathrm{dist}(\Sigma,\partial B_\varrho(x_0))$, then we have $\Sigma\subseteq B_{\varrho-d}(x_0)$. For every $0<\delta<d$, we define the cut-off function 
	\[
	\eta_\delta(x) := \min\left\{\left(\dfrac{(\varrho- \delta) - |x-x_0|}{d- \delta}\right)_{+},\, 1\right\}. 
	\]
	In particular, $\eta_\delta$ is a Lipschitz function with compact support in $B_{\varrho}(x_0)$ and 
	\[
	\eta_\delta = 1 \quad \mbox{ on } \overline{B_{\varrho-d}(x_0)}, \qquad |\nabla \eta_\delta| \leq \frac{1}{d - \delta}, \qquad \eta_\delta = 0 \quad \mbox{ on } \mathbb{R}^N \setminus B_{\varrho- \delta}(x_0).
	\]  
We take $u\in C^\infty_0(\mathcal{K}_\varrho(x_0))$ such that $u\ge 1$ on $\Sigma$.
Then, by the definition of relative capacity and Minkowski's inequality, we have \[
	\begin{split}
		\Big(\mathrm{cap}_p(\Sigma; B_{\varrho}(x_0))\Big)^{\frac{1}{p}} &\leq \|\nabla(u\,\eta_\delta)\|_{L^p(B_{\varrho}(x_0))} \\
		&\leq \frac{1}{d-\delta}\,\|u\|_{L^p(\mathcal{K}_{\varrho}(x_0))} + \|\nabla u\|_{L^p(\mathcal{K}_{\varrho}(x_0))} \\
		&\leq \left(\frac{1}{d - \delta}\,\frac{1}{\lambda_p(\mathcal{K}_{\varrho}(x_0))^{\frac{1}{p}}} + 1\right)\,\|\nabla u\|_{L^p(\mathcal{K}_{\varrho}(x_0))},
	\end{split}
	\] 
	where in the last line we used Poincar\'e inequality on the bounded set $\mathcal{K}_{\varrho}(x_0)$. By sending $\delta$ to $0$ and by the arbitrariness of $u$, we then obtain the claimed estimate.
\end{proof}
The previous result has a sort of converse.
\begin{lemma} \label{lm:ball-VS-body}
Let $1\le p<\infty$, for every compact set $\Sigma\subseteq \mathcal{K}_\varrho(x_0)$ we have
\[
\mathrm{cap}_p(\Sigma; \mathcal{K}_{\varrho}(x_0)) \leq  \left(\frac{\varrho}{\mathrm{dist}(\Sigma,\partial \mathcal{K}_\varrho(x_0))}\,\frac{R_\mathcal{K}}{\lambda_{p}(B_1)^{\frac{1}{p}}} + 1\right)^p\,\mathrm{cap}_p(\Sigma; B_{R_\mathcal{K}\varrho}(x_0)).
\]
\end{lemma}
\begin{proof}
We first observe that if $\Sigma\subseteq \mathcal{K}_\varrho(x_0)$, then we have $\Sigma\subseteq B_{R_\mathcal{K}\varrho}(x_0)$, as well. This simply follows from the fact that
\[
\mathcal{K}_\varrho(x_0)\subseteq B_{R_\mathcal{K}\varrho}(x_0),
\]
just by recalling the definition of $R_\mathcal{K}$.
\par
Let $u \in C^\infty_0(B_{R_\mathcal{K} \varrho}(x_0))$ be such that $u \ge 1$ on $\Sigma$. For every $0 < \delta <d:=\mathrm{dist}(\Sigma,\partial \mathcal{K}_\varrho(x_0))$, we introduce the cut-off function $\xi_\delta$
	given by \[
	\xi_\delta(x) = \min\left\{\frac{1}{d - \delta}\,\mathrm{dist}\Big(x;\,\, \partial \mathcal{K}_{\varrho - \delta}(x_0)\Big),\,\, 1\right\}, \qquad \mbox{ for } x \in \mathcal{K}_{\varrho - \delta}(x_0),  
	\]
	and extended by zero over the whole $\mathbb{R}^N$. This is a Lipschitz function with compact support in $\mathcal{K}_\varrho(x_0)$ such that $\xi_\delta=1$ on $\Sigma$. Thus, by the definition of relative capacity and Minkowski's inequality, we obtain	\[
	\begin{split}
		\Big( \mathrm{cap}_p(\Sigma; \mathcal{K}_\varrho(x_0))\Big)^{\frac{1}{p}} &\leq \|\nabla \left(u\,\xi_\delta\right)\|_{L^p(\mathcal{K}_\varrho(x_0))} \\ 
		&\leq \frac{1}{d - \delta}\,\|u\|_{L^p(B_{R_\mathcal{K}\varrho})} + \|\nabla u\|_{L^p(B_{R_\mathcal{K}\varrho}(x_0))} \\
		&\leq \left(\frac{\varrho}{d - \delta}\,\frac{R_\mathcal{K}}{\lambda_{p}(B_1)^{\frac{1}{p}}} + 1\right)\,\|\nabla u\|_{L^p(B_{R_\mathcal{K}\varrho}(x_0))},
	\end{split}
	\] 
	where in the last line we used Poincar\'e inequality for $B_{R_{\mathcal{K}}\varrho}(x_0)$. By the arbitrariness of $u$ and $\delta$, we then have  the desired estimate.
\end{proof}
We have the following result, which the quantities $R_{p,\gamma}(\Omega;\mathcal{K})$ with $R_{p,\gamma}(\Omega)$. This is a generalization of \cite[Proposition 2.2.5]{Boz}, contained in the Ph.D. thesis of the second author.
\begin{proposition} 
	Let $1 \leq p \leq N$ and let $\Omega \subseteq \mathbb{R}^N$ be an open set. For every standard body $\mathcal{K} \subseteq \mathbb{R}^N$, there exist two constants $0 < c \leq 1\le d$, both depending only on $N$, $p$ and $\mathcal{K}$,  such that we have 
	\begin{equation} \label{equiv-innercubic-capin}
		R_{p, c\cdot\gamma}(\Omega; \mathcal{K}) \leq R_{p, \gamma}(\Omega), \qquad\text{for every}\ 0<\gamma<1,
	\end{equation}
	and
	\begin{equation} \label{equiv-innercubic-capin2}	
	R_{p, \gamma}(\Omega) 	\leq R_\mathcal{K}\,R_{p, d\cdot\gamma}(\Omega; \mathcal{K}),\qquad\text{for every}\ 0<\gamma<\frac{1}{d},
	\end{equation}
	where $R_\mathcal{K} = \max_{x \in \partial \mathcal{K}} |x|. $
\end{proposition} 
\begin{proof}
	We can suppose that $R_{p, \gamma}(\Omega) < +\infty$, otherwise the leftmost inequality in \eqref{equiv-innercubic-capin} is trivial. We set 
	\begin{equation}
	\label{c}
		c=c(N,p,\mathcal{K}):= \left(\frac{\mathrm{cap}_p(\overline{B_1}; B_2)}{\mathrm{cap}_p(\overline{\mathcal{K}}; \mathcal{K}_{2})}\right)\,\left(\frac{2}{\lambda_{p}(\mathcal{K})^{\frac{1}{p}}} + 1\right)^{-p},
	\end{equation}
and observe that $c\le 1$, by Lemma \ref{lm:cap_palleVSbody} with $x_0=0$, $\Sigma=\overline{B_1}$ and $\varrho=2$.
	Let $r > R_{p, \gamma}(\Omega)$, by definition of capacitary inradius we then have
	 \begin{equation} \label{eqn:def-capin-prop}
		\mathrm{cap}_p(\overline{B_r(x_0)} \setminus \Omega; B_{2r}(x_0)) > \gamma\,\mathrm{cap}_p(\overline{B_r(x_0)}; B_{2r}(x_0)), \qquad \mbox{ for every } x_0 \in \mathbb{R}^N,
	\end{equation}
By using \eqref{francesco} with $\Sigma=\overline{B_r(x_0)}\setminus \Omega$ and $\varrho=2\,r$, we also get	
	\begin{equation*}
		\begin{split}
			\mathrm{cap}_p(\overline{B_r(x_0)}\setminus \Omega; B_{2r}(x_0)) &\leq \left(\frac{2}{\lambda_{p}(\mathcal{K})^{\frac{1}{p}}} + 1\right)^p\,\mathrm{cap}_p(\overline{B_r(x_0)}\setminus \Omega; \mathcal{K}_{2r}(x_0)), \\
			&\leq \left(\frac{2}{\lambda_{p}(\mathcal{K})^{\frac{1}{p}}} + 1\right)^p\,\mathrm{cap}_p(\overline{\mathcal{K}_r(x_0)}\setminus \Omega; \mathcal{K}_{2r}(x_0)).
		\end{split}
	\end{equation*}
	Observe that we also used \eqref{monotone}, thanks to the fact that $\overline{B_r(x_0)}\setminus \Omega\subseteq \overline{\mathcal{K}_r(x_0)}\setminus \Omega$.
This, combined with \eqref{eqn:def-capin-prop} gives that
\[
\gamma\,\left(\frac{2}{\lambda_{p}(\mathcal{K})^{\frac{1}{p}}} + 1\right)^{-p}\,\mathrm{cap}_p(\overline{B_r(x_0)}; B_{2r}(x_0))\le \mathrm{cap}_p(\overline{\mathcal{K}_r(x_0)}\setminus \Omega; \mathcal{K}_{2r}(x_0)).
\]
On the other hand, by the scaling properties of the relative capacity, we have 
\[
	\mathrm{cap}_p(\overline{B_r(x_0)}; B_{2r}(x_0)) = \left(\frac{\mathrm{cap}_p(\overline{B_1}; B_2)}{\mathrm{cap}_p(\overline{\mathcal{K}}; \mathcal{K}_{2})}\right)\,\mathrm{cap}_p(\overline{\mathcal{K}_r(x_0)}; \mathcal{K}_{2r}(x_0)).
\]	
Thus, by recalling the definition \eqref{c} of $c$, we have obtained
	\[
	c\cdot\gamma\,\mathrm{cap}_p(\overline{\mathcal{K}_r}; \mathcal{K}_{2r})< \mathrm{cap}_p(\overline{\mathcal{K}_r}\setminus \Omega; \mathcal{K}_{2r}),
	\]
By the definition of capacitary inradius relative to $\mathcal{K}$, this shows that $r>R_{p,\gamma}(\Omega;\mathcal{K})$. By the arbitrariness of $r$, we can finally obtain \eqref{equiv-innercubic-capin}.
\vskip.2cm\noindent
Let us choose 
\begin{equation*}
		d = d(N, p , \mathcal{K}) =	\max\left\{1, R_{\mathcal{K}}^{N- p}\,\Bigg(\dfrac{\mathrm{cap}_p(\overline{B_{1}}; B_{2})}{\mathrm{cap}_p(\overline{\mathcal{K}}; \mathcal{K}_2)}\Bigg)\,\left(\frac{2}{\mathrm{dist}(\overline{\mathcal{K}},\partial \mathcal{K}_2)}\,\frac{R_\mathcal{K}}{\lambda_{p}(B_1)^{\frac{1}{p}}} + 1\right)^{p}\right\},
	\end{equation*}
and take $0<\gamma<1/d$. We can suppose that $R_{p, d\cdot\gamma}(\Omega; \mathcal{K}) < +\infty$. We take any $r > R_{p, d\cdot\gamma}(\Omega; \mathcal{K})$, so that
	\begin{equation} \label{cavallino}
		\mathrm{cap}_p(\overline{\mathcal{K}_r(x_0)} \setminus \Omega; \mathcal{K}_{2r}(x_0)) > d\cdot\gamma\,\mathrm{cap}_p(\overline{\mathcal{K}_r(x_0)}; \mathcal{K}_{2r}(x_0)), \qquad \mbox{ for every } x_0 \in \mathbb{R}^N.
	\end{equation}
For the leftmost term, by using Lemma \ref{lm:ball-VS-body} with $\Sigma=\overline{\mathcal{K}_r(x_0)}\setminus \Omega$ and $\varrho=2\,r$, we have\footnote{We use that
\[
\mathrm{dist}(\overline{\mathcal{K}_r(x_0)}\setminus \Omega;\partial \mathcal{K}_{2r}(x_0))\ge \mathrm{dist}(\overline{\mathcal{K}_r(x_0)};\partial \mathcal{K}_{2r}(x_0))=r\,\mathrm{dist}(\overline{\mathcal{K}};\partial \mathcal{K}_{2}).
\]}
	\begin{equation} \label{eqn:finalmente}
	\begin{split}
		\mathrm{cap}_p(\overline{\mathcal{K}_r(x_0)}\setminus \Omega; \mathcal{K}_{2r}(x_0)) &\leq  \left(\frac{2}{\mathrm{dist}(\overline{\mathcal{K}},\partial \mathcal{K}_2)}\,\frac{R_\mathcal{K}}{\lambda_{p}(B_1)^{\frac{1}{p}}} + 1\right)^p\,\mathrm{cap}_p(\overline{\mathcal{K}_r(x_0)}\setminus \Omega; B_{2R_\mathcal{K} r}(x_0))\\
&\le \left(\frac{2}{\mathrm{dist}(\overline{\mathcal{K}},\partial \mathcal{K}_2)}\,\frac{R_\mathcal{K}}{\lambda_{p}(B_1)^{\frac{1}{p}}} + 1\right)^p\,\mathrm{cap}_p(\overline{B_{R_\mathcal{K} r}(x_0)} \setminus \Omega; B_{2R_\mathcal{K} r}(x_0)).
	\end{split}
	\end{equation} 
The second inequality follows again from \eqref{monotone}, since we have $\mathcal{K}_r(x_0)\subseteq B_{R_\mathcal{K} r}(x_0)$. Observe that by scaling
\[
\begin{split}
\mathrm{cap}_p(\overline{\mathcal{K}_r(x_0)}; \mathcal{K}_{2r}(x_0))&=r^{N-p}\,\frac{\mathrm{cap}_p(\overline{\mathcal{K}}; \mathcal{K}_2)}{\mathrm{cap}_p(\overline{B_1},B_2)}\,\mathrm{cap}_p(\overline{B_1},B_2)\\
&=\left(\frac{1}{R_\mathcal{K}}\right)^{N-p}\,\frac{\mathrm{cap}_p(\overline{\mathcal{K}}; \mathcal{K}_2)}{\mathrm{cap}_p(\overline{B_1},B_2)}\,\mathrm{cap}_p(\overline{B_{R_\mathcal{K}r}(x_0)},B_{2R_\mathcal{K}r}(x_0)).
\end{split}
\]
Thus, by recalling the definition of $d$, we get in particular
\[
\mathrm{cap}_p(\overline{\mathcal{K}_r(x_0)}; \mathcal{K}_{2r}(x_0))\ge \frac{1}{d}\,\left(\frac{2 R_\mathcal{K}}{\lambda_{p}(B_1)^{\frac{1}{p}}} + 1\right)^{p}\,\mathrm{cap}_p(\overline{B_{R_\mathcal{K}r}},B_{2R_\mathcal{K}r}).
\]
Together with \eqref{cavallino} and \eqref{eqn:finalmente}, this implies that 
	\begin{equation*}
		\mathrm{cap}_p\Big(\overline{B_{R_\mathcal{K} r}(x_0)} \setminus \Omega; B_{2R_\mathcal{K} r}(x_0)\Big) > d\cdot\gamma\,\mathrm{cap}_p\Big(\overline{B_{R_\mathcal{K} r}(x_0)}; B_{2R_\mathcal{K} r}(x_0)\Big), \qquad \mbox{ for every } x_0 \in \mathbb{R}^N.
	\end{equation*}
	Therefore, by the definition of capacitary inradius  we have 
\[
		R_\mathcal{K}\,r \geq R_{p, \gamma}(\Omega).
\]
By arbitrariness of $r>R_{p,\gamma}(\Omega;\mathcal{K})$, we obtain \eqref{equiv-innercubic-capin2}.
\end{proof} 
\begin{open}
Prove that there exist two constants $a=a(N,p,\mathcal{K},\gamma),b=b(N,p,\mathcal{K},\gamma)>0$ such that
\[
a\,R_{p,\gamma}(\Omega;\mathcal{K})\le R_{p,\gamma}(\Omega)\le b\,R_{p,\gamma}(\Omega;\mathcal{K}),\qquad \text{\it for every}\ 0<\gamma<1.
\]
This would permit to obtain a two-sided estimate on $\lambda_p(\Omega)$ and, more generally, on $\lambda_{p,q}(\Omega)$ with $q>p$ subcritical, in terms on $R_{p,\gamma}(\Omega;\mathcal{K})$, no matter the shape $\mathcal{K}$ and the negligibility parameter $\gamma$. For the particular case $\mathcal{K}=(-1,1)^N$, such a result would extend \cite[Theorem 15.4.1]{Maz} (in the case of first order Sobolev spaces), by dropping the restriction $\gamma\le \gamma_0(N,p)$ there taken (see \cite[equation (14.1.2)]{Maz}).
\end{open}

\section{Different capacities}
\label{sec:4}
In this section, rather than changing the basic ``shape'' used to compute the capacity, we will vary the concept of capacity itself, used to define the capacitary inradius. In particular, we want to compare our definition of $R_{p,\gamma}$ with the ones contained in \cite{Ga2} and \cite{MS}. 

\subsection{Maz'ya-Shubin inradius}
We start with the capacitary inradius defined in \cite{MS}.
 To this aim, we need at first to recall the following notion of $p-$capacity.
\begin{definition} \label{def:abs-hom-cap}
Let $1\le p<N$ or $p=N=1$, for every $\Sigma\subseteq\mathbb{R}^N$ compact set we define its {\it absolute homogeneous $p-$capacity} as
\[
\mathrm{cap}_p(\Sigma)=\inf_{\varphi\in C^\infty_0(\mathbb{R}^N)} \left\{\int_{\mathbb{R}^N} |\nabla \varphi|^p\,dx\, :\, \varphi\ge 1\ \text{on}\ \Sigma\right\}.
\]
Observe that this coincides with $\mathrm{cap}_p(\Sigma;\mathbb{R}^N)$ in our previous notation, i.e. this is the $p-$capacity of $\Sigma$ relative to the whole space $\mathbb{R}^N$.
\end{definition}
\begin{remark}
\label{rem:accia}
The restriction $p<N$ is unavoidable in the previous definition, at least for $N\ge 2$. Indeed, in this situation we have
\[
\mathrm{cap}_p(\Sigma)=0,\qquad \text{for every}\ p\ge N\ \text{and every}\ \Sigma\subseteq\mathbb{R}^N \ \text{compact set}.
\]
For $p>N$, this follows by a simple scaling argument. The borderline case $p=N\ge 2$ is slightly more delicate, because of the scale invariance of the $N-$Dirichlet integral (see for example \cite[pages 148--149]{Maz}).
\end{remark}
We quantitatively compare the absolute capacity with the relative one. To this aim, we need to recall the definition of sharp constant in Sobolev's inequality for $p<N$, i.e.
\[
\mathcal{S}_{N,p}=\sup_{\varphi\in C^\infty_0(\mathbb{R}^N)} \left\{\left(\int_{\mathbb{R}^N} |\varphi|^{p^*}\,dx\right)^\frac{p}{p^*}\, :\, \|\nabla \varphi\|_{L^p(\mathbb{R}^N)}=1\right\},\qquad p^*=\frac{N\,p}{N-p}.
\]
\begin{lemma}[Absolute VS. relative]
\label{lm:relative_absolute}
Let $1\le p<N$ and let $\Sigma\subseteq\mathbb{R}^N$ be a compact set. If $\Sigma\Subset E$ with $E\subseteq \mathbb{R}^N$ open bounded set, we have
\[
\mathrm{cap}_p(\Sigma)\le \mathrm{cap}_p(\Sigma;E)\le 2^{p-1}\,\left(1+\frac{|E|^\frac{p}{N}}{(\mathrm{dist}(\Sigma,\partial E))^p}\,\mathcal{S}_{N,p}\right)\,\mathrm{cap}_p(\Sigma).
\]
\end{lemma}
\begin{proof}
The leftmost inequality trivially follows from the fact that in the definition of $\mathrm{cap}_p(\Sigma)$ we perform the infimum on a larger class of functions. For the rightmost one, let us take 
$\varphi\in C^\infty_0(\mathbb{R}^N)$ such that $\varphi\ge 1$ on $\Sigma$. For every $0<\delta<\mathrm{dist}(\Sigma,\partial E)$, we take the following Lipschitz cut-off function with compact support in $E$, i.e.
\begin{equation}
\label{telostronco}
\eta_\delta(x)=\min\left\{\frac{(\mathrm{dist}(x,\partial E)-\delta)_+}{\mathrm{dist}(\Sigma,\partial E)-\delta},1\right\}.
\end{equation}
Observe that
\[
0\le \eta_\delta\le 1,\qquad \eta_\delta\equiv 1\ \text{on}\ \Sigma,\qquad |\nabla \eta_\delta|\le \frac{1}{\mathrm{dist}(\Sigma,\partial E)-\delta}.
\]
The function $\varphi\,\eta_\delta$ is feasible for the minimization problem which defines $\mathrm{cap}_p(\Sigma;E)$. Thus, we obtain
\[
\begin{split}
\mathrm{cap}_p(\Sigma;E)&\le 2^{p-1}\,\int_E |\nabla \varphi|^p\,\eta_\delta^p\,dx+2^{p-1}\,\int_E |\nabla \eta_\delta|^p\,|\varphi|^p\,dx\\
&\le 2^{p-1}\,\int_{\mathbb{R}^N} |\nabla \varphi|^p\,dx+\frac{2^{p-1}}{(\mathrm{dist}(\Sigma,\partial E)-\delta)^p}\,\int_{E} |\varphi|^p\,dx.
\end{split}
\] 
To control the last term, we combine H\"older's and Sobolev's inequalities as follows
\[
\int_{E} |\varphi|^p\,dx\le |E|^{1-\frac{p}{p^*}}\,\left(\int_{E} |\varphi|^{p^*}\,dx\right)^\frac{p}{p^*}\le |E|^{1-\frac{p}{p^*}}\,\mathcal{S}_{N,p}\,\int_{\mathbb{R}^N} |\nabla \varphi|^p\,dx.
\]
In conclusion, we obtain
\[
\mathrm{cap}_p(\Sigma;E)\le 2^{p-1}\,\left(1+\frac{|E|^\frac{p}{N}}{(\mathrm{dist}(\Sigma,\partial E)-\delta)^p}\,\mathcal{S}_{N,p}\right)\,\int_{\mathbb{R}^N} |\nabla \varphi|^p\,dx.
\]
By arbitrariness of $\delta$ and $\varphi$, we get the conclusion.
\end{proof}
\begin{remark}
As one can easily see by inspecting the previous proof, the boundedness assumption on $E$ is not really needed. The same proof works under the assumption $|E|<+\infty$.
\end{remark}
For $N\ge 2$, we see that the rightmost inequality of Lemma \ref{lm:relative_absolute} gets spoiled as $p$ goes to $N$. Indeed, we recall that
\[
\mathcal{S}_{N,p}\sim \frac{1}{(N-p)^{p-1}},\qquad \text{as}\ p\nearrow N,
\] 
see for example \cite{TaG}.
This is accordance with Remark \ref{rem:accia}, thus the equivalence of the two capacities {\it must} cease to be true in this limit case.
\par
The case $p=N=1$ is however exceptional: as simple and useless as it may seem, it deserves a separate statement.
\begin{lemma}
\label{lemma:relative_absolute1d}
Let $\Sigma\subseteq\mathbb{R}$ be a compact set. If $\Sigma\Subset E$ with $E\subseteq \mathbb{R}$ open bounded set, we have
\[
\mathrm{cap}_1(\Sigma)\le \mathrm{cap}_1(\Sigma;E)\le \left(1+\frac{|E|}{\mathrm{dist}(\Sigma,\partial E)}\,\frac{1}{2}\right)\,\mathrm{cap}_1(\Sigma).
\]
\end{lemma}
\begin{proof}
We can proceed as in the proof of Lemma \ref{lm:relative_absolute}. The only difference is in the estimate of
\[
\int_E |\varphi|\,dx\le |E|\,\|\varphi\|_{L^\infty(\mathbb{R})}\le \frac{|E|}{2}\,\int_{\mathbb{R}} |\varphi'|\,dx.
\]
This is enough to conclude.
\end{proof}
\begin{definition} \label{def:ms-capin}
Let $1\le p<N$ or $p=N=1$ and let $\gamma\in(0,1)$. For an open set $\Omega\subseteq\mathbb{R}^N$ we define its {\it Maz'ya-Shubin capacitary inradius} by 
\[
R^{\rm MS}_{p,\gamma}(\Omega):=\sup\Big\{r>0\, :\, \exists x_0\in\mathbb{R}^N\ \text{such that}\ \mathrm{cap}_p\left(\overline{B_r(x_0)}\setminus\Omega\right)\le \gamma\,\mathrm{cap}_p\left(\overline{B_r(x_0)}\right)\Big\}.
\]
Observe that the only difference with our capacitary inradius $R_{p,\gamma}(\Omega)$ is the use of the absolute homogeneous $p-$capacity in formulating the $(p,\gamma)-$negligibility condition, in place of the relative one.
\end{definition}
\begin{proposition}
	Let $1\le p<N$ or $p=N=1$, then there exist two constants $\alpha = \alpha(N,p)\le 1$ and $\beta=\beta(N,p)\ge 1$ such that for every $\Omega\subseteq\mathbb{R}^N$ open set we have 
	\[
	R^{\rm MS}_{p,\alpha\cdot\gamma}(\Omega)\le R_{p,\gamma}(\Omega),\qquad\text{for every}\ 0<\gamma<1,
	\]
	and
	\[
R_{p,\gamma}(\Omega)	\le R^{\rm MS}_{p,\beta \cdot\gamma}(\Omega),\qquad\text{for every}\ 0<\gamma<\frac{1}{\beta}.
	\]
Moreover, for $N\ge 2$ we have
\[
\lim_{p\nearrow N}\alpha=0\qquad\text{and}\qquad \lim_{p\nearrow N}\beta=+\infty.
\]
\end{proposition}
\begin{proof}
We can confine ourselves to consider the case $N\ge 2$, the case $N=1$ can be handled in exactly the same manner, by using Lemma \ref{lemma:relative_absolute1d} in place of Lemma \ref{lm:relative_absolute}.
\par
	In order to prove the first inequality, we take $0<\gamma<1$ and we can suppose that $R_{p,\gamma}(\Omega) < +\infty$. Let $r > R_{p,\gamma}(\Omega)$, this implies that 
	\[
		\mathrm{cap}_p(\overline{B_r(x_0)} \setminus \Omega; B_{2r}(x_0)) > \gamma\, \mathrm{cap}_p(\overline{B_r(x_0)}; B_{2r}(x_0)), \qquad \text{for every}\ x_0 \in \mathbb{R}^N. 
	\]	
	By using Lemma \ref{lm:relative_absolute}, this in particular entails that   
	\[
	\gamma\,\mathrm{cap}_p\left(\overline{B_{r}(x_0)}\right) < 2^{p-1}\,\left(1+\left(2\omega_N\right)^\frac{p}{N}\,\mathcal{S}_{N,p}\right)\,\mathrm{cap}_p\left(\overline{B_r(x_0)}\setminus \Omega\right), \qquad \mbox{ for every } x \in \mathbb{R}^N.
	\]
	Thus, by the arbitrariness of $r$ and the Definition \ref{def:ms-capin}, we can infer that
	\[
	R^{\rm MS}_{p,\alpha \cdot \gamma}(\Omega) \le R_{p, \gamma}(\Omega),
	\]
	where $\alpha$ is given by 
	\[
	\alpha:= 2^{1-p}\,\left(1+ (2\,\omega_N)^{\frac{p}{N}}\,\mathcal{S}_{N,p}\right)^{-1}. 
	\]
	Observe that $\alpha\le 1$ and for $N\ge 2$ it converges to $0$, as $p$ goes to $N$.
	\vskip.2cm\noindent
	In order to prove the second inequality, we set 
	\[
	\beta:= \frac{\mathrm{cap}_p(\overline{B_1}; B_2)}{\mathrm{cap}_p(\overline{B_1})}=\left(1-\left(\dfrac{1}{2}\right)^\frac{N-p}{p-1}\right)^{1-p},\qquad \text{if}\ 1<p<N,
	\] 
	and $\beta=1$ in the case $p=1$.
	We take $0<\gamma<1/\beta$ and suppose that $R^{\rm MS}_{p, \beta\cdot\gamma}(\Omega) < +\infty$, otherwise there is nothing to prove. We take any $r > R^{\rm MS}_{p, \beta\cdot\gamma}(\Omega)$, then thanks to the definition of $\beta$ we have
	\begin{equation*} \label{eqn:inner-cubic-def}
		\begin{split}
		\mathrm{cap}_p(\overline{B_r(x_0)} \setminus \Omega; B_{2r}(x_0)) &\ge \mathrm{cap}_p\left(\overline{B_r(x_0)} \setminus \Omega\right) \\
		&> \beta\,\gamma\,\mathrm{cap}_p\left(\overline{B_r(x_0)}\right) = \gamma\, \mathrm{cap}_p(\overline{B_r(x_0)}; B_{2r}(x_0)), 
	\end{split}
	\end{equation*}
	for every $x_0 \in \mathbb{R}^N$, where the second inequality comes from the Definition \ref{def:ms-capin}. In particular, we have that $r > R_{p, \gamma}(\Omega)$. By the arbirtrariness of $r$, we eventually conclude that $R^{\rm MS}_{p, \beta\cdot\gamma}(\Omega) \ge R_{p, \gamma}(\Omega)$.  
\end{proof}
\begin{open}
For $1\le p<N$, prove that there exist two constants $c_1=c_1(N,p,\gamma),c_2=c_2(N,p,\gamma)>0$ such that
\[
c_1\,R_{p,\gamma}^{\rm MS}(\Omega)\le R_{p,\gamma}(\Omega)\le c_2\,R_{p,\gamma}^{\rm MS}(\Omega),\qquad \text{\it for every}\ 0<\gamma<1.
\]
For $p=2$, this can be simply obtained by combining \eqref{capin} with the two-sided estimate by Maz'ya and Shubin, i.e. \cite[Theorem 1.1]{MS}. For a general $p$, one could proceed similarly: one then needs a generalization of their result, i.e. a two-sided estimate on $\lambda_p(\Omega)$ in terms of $R^{\rm MS}_{p,\gamma}(\Omega)$.
\end{open}

\subsection{Gallagher inradius}
We need at first the following alternative notion of absolute capacity.
\begin{definition}
Let $1\le p<\infty$, for every $\Sigma\subseteq\mathbb{R}^N$ compact set we define its {\it absolute inhomogeneous $p-$capacity} as 
\[
\mathscr{C}_p(\Sigma)=\inf_{\varphi\in C^\infty_0(\mathbb{R}^N)} \left\{\int_{\mathbb{R}^N} |\nabla \varphi|^p\,dx+\int_{\mathbb{R}^N} |\varphi|^p\,dx\, :\, \varphi\ge 1\ \text{on}\ \Sigma\right\}.
\]
\end{definition}
Observe that the two norms appearing in the minimization problem scale differently. Accordingly, this quantity does not enjoy a scaling relation\footnote{This explains the name {\it inhomogeneous}.}. Nevertheless, we have the following simple relation for the $p-$capacity of rescaled sets.
\begin{lemma}
\label{lm:nonno}
Let $1\le p<\infty$, for every $\Sigma\subseteq\mathbb{R}^N$ compact set and every $t>0$, we have
\[
\min\{t^{N-p},t^{N}\}\,\mathscr{C}_p(\Sigma)\le \mathscr{C}_p(t\,\Sigma)\le \max\{t^{N-p},\,t^{N}\}\,\mathscr{C}_p(\Sigma).
\]
\end{lemma}
\begin{proof}
Let $\varphi\in C^\infty_0(\mathbb{R}^N)$ be such that $\varphi\ge 1$ on $\Sigma$. Then the rescaled function $\varphi_t(x)=\varphi(x/t)$ belongs to $C^\infty_0(\mathbb{R}^N)$ and is such that $\varphi_t\ge 1$ on $t\,\Sigma$. Thus, with a change of variable we get
\[
\begin{split}
\mathscr{C}_p(t\,\Sigma)&\le t^{N-p}\,\int_{\mathbb{R}^N} |\nabla\varphi|^p\,dx+t^{N}\int_{\mathbb{R}^N} |\varphi|^p\,dx\\
&\le \max\{t^{N-p},t^N\}\,\left[\int_{\mathbb{R}^N} |\nabla\varphi|^p\,dx+\int_{\mathbb{R}^N} |\varphi|^p\,dx\right].		
\end{split}
\]
By taking the infimum over $\varphi$, we get the upper bound. 
\par
For the lower bound, we proceed in a similar way. Let $\varphi\in C^\infty_0(\mathbb{R}^N)$ be such that $\varphi\ge 1$ on $t\,\Sigma$. Then the rescaled function $\varphi_t(x)=\varphi(t\,x)$ belongs to $C^\infty_0(\mathbb{R}^N)$ and is such that $\varphi_t\ge 1$ on $\Sigma$. Similarly as before, we get
\[
\begin{split}
\mathscr{C}_p(\Sigma)&\le t^{p-N}\,\int_{\mathbb{R}^N} |\nabla\varphi|^p\,dx+t^{-N}\int_{\mathbb{R}^N} |\varphi|^p\,dx\\
&\le \max\{t^{p-N},t^{-N}\}\,\left[\int_{\mathbb{R}^N} |\nabla\varphi|^p\,dx+\int_{\mathbb{R}^N} |\varphi|^p\,dx\right].		
\end{split}
\]
By arbitrariness of $\varphi$, we conclude.
\end{proof}
The comparison between this $p-$capacity and the relative one is contained for example in \cite[Proposition 13.1.1/2]{Maz}. We repeat its simple proof, so to make precise the constants appearing in the estimate.
\begin{lemma}
\label{lm:ammazya!}
Let $1\le p<\infty$ and let $\Sigma\subseteq\mathbb{R}^N$ be a compact set. 
If $\Sigma\Subset E$ with $E\subseteq \mathbb{R}^N$ open bounded set, we have 
\[
\frac{\lambda_p(E)}{1+\lambda_p(E)}\,\mathscr{C}_p(\Sigma)\le \mathrm{cap}_p(\Sigma;E)\le 2^{p-1}\,\max\left\{1,\frac{1}{(\mathrm{dist}(\Sigma,\partial E))^p}\right\}\,\mathscr{C}_p(\Sigma).
\]
\end{lemma}
\begin{proof}
We take $\varphi\in C^\infty_0(E)$ such that $\varphi\ge 1$ on $\Sigma$. Thus, by definition of $\mathscr{C}_p(\Sigma)$ and Poincar\'e inequality, we get
\[
\begin{split}
\mathscr{C}_p(\Sigma)&\le \int_{\mathbb{R}^N} |\nabla \varphi|^p\,dx+\int_{\mathbb{R}^N} |\varphi|^p\,dx\le \left(1+\frac{1}{\lambda_p(E)}\right)\, \int_{\mathbb{R}^N} |\nabla \varphi|^p\,dx.
\end{split}
\]
By arbitrariness of $\varphi$, we get the first inequality. 
\par
For the second inequality, we take $\varphi\in C^\infty_0(\mathbb{R}^N)$ such that $\varphi\ge 1$ on $\Sigma$ and the cut-off function $\eta_\delta$ defined in \eqref{telostronco}. By using $\varphi\,\eta_\delta$, we get
\[
\mathrm{cap}_p(\Sigma;E)\le 2^{p-1}\,\int_{\mathbb{R}^N} |\nabla\varphi|^p\,\eta_\delta^p\,dx+2^{p-1}\,\int_{\mathbb{R}^N} |\nabla\eta_\delta|^p\,|\varphi|^p\,dx.
\]
By using the properties of $\eta_\delta$ and recalling the definition of $\mathscr{C}_p(\Sigma)$, we easily get the conclusion.
\end{proof}
\begin{remark}
An inspection of the previous proof reveals again that the boundedness assumption on $E$ is not really needed. It would be sufficient to suppose for example that $E$ supports the $p-$Poincar\'e inequality. This condition is somehow necessary, as well: take for example $E=\mathbb{R}^N$ and $\Sigma=\overline{B_r}$. Then, we have 
\[
\mathrm{cap}_p(\overline{B_r};\mathbb{R}^N)=\mathrm{cap}_p(\overline{B_r}).
\]
For $p\ge N\ge 2$ we have already said that this is zero, for every $r>0$. For $1\le p<N$, we have 
\[
\mathrm{cap}_p(\overline{B_r})=N\,\omega_N\, \left(\frac{N-p}{p-1}\right)^{p-1}\,r^{N-p},
\]
it is sufficient to take the limit as $R$ goes to $\infty$ in \eqref{eqn:cap-ball0}, \eqref{eqn:cap-ball}.
On the other hand, by proceeding as in the proof of Lemma \ref{lemma:2}, we easily see that
\[
\mathscr{C}_p(\overline{B_r})\ge |\overline{B_r}|=\omega_N\,r^N.
\]
Thus, in any case we get that
\[
\lim_{r\to+\infty} \frac{\mathrm{cap}_p(\overline{B_r})}{\mathscr{C}_p(\overline{B_r})}=0.
\]
\end{remark}
Based on the previous $p-$capacity, in \cite{Ga1, Ga2} the following capacitary inradius is introduced. The definition is inspired by a related quantity defined by Souplet in \cite[Section 2]{Sou}, using Lebesgue measure in place of $p-$capacity.
\begin{definition}
Let $1\le p<\infty$. For an open set $\Omega\subseteq\mathbb{R}^N$ we defined its {\it Gallagher capacitary inradius} by 
\[
R^{\rm G}_{p}(\Omega):=\sup\Big\{r>0\, :\, \forall \varepsilon>0,\ \exists x_0\in\mathbb{R}^N\ \text{such that}\ \mathscr{C}_p\left(\overline{B_r(x_0)}\setminus\Omega\right)< \varepsilon\Big\}.
\]
Observe that we have also included the case $p=1$, which was not considered in \cite{Ga2}.
\end{definition}
Since this definition is based on an inhomogeneous concept of $p-$capacity, it is not clear whether $R^{\rm G}_p$ enjoys some scaling rule or not. 
This is actually the case, as shown in the following
\begin{lemma}
Let $1\le p<\infty$ and let $\Omega \subseteq \mathbb{R}^N$ be an open set. Then for every $t>0$ we have
\[
R^{\rm G}_p(t\,\Omega)= t\,R^{\rm G}_p(\Omega)
\]
\end{lemma}
\begin{proof}
For every $\delta>0$, there exists $R^{\rm G}_p(t\,\Omega)-\delta<r_\delta<R^{\rm G}_p(t\,\Omega)$. Thus, for every $\varepsilon>0$ there exists a point $x_0\in\mathbb{R}^N$ such that
\[
\mathscr{C}_p(\overline{B_{r_\delta}(x_0)}\setminus(t\,\Omega))<\varepsilon.
\]
Observe that
\[
\overline{B_{r_\delta}(x_0)}\setminus(t\,\Omega)=\overline{t\,B_{{r_\delta}/t}(x_0/t)}\setminus(t\,\Omega)=t\,\left(\overline{B_{{r_\delta}/t}(x_0/t)}\setminus\Omega\right),
\]
thus we have in particular
\[
\mathscr{C}_p\left(\overline{B_{{r_\delta}/t}(x_0/t)}\setminus\Omega\right)<\frac{1}{\min\{t^N,\,t^{N-p}\}}\,\varepsilon,
\]
thanks to Lemma \ref{lm:nonno}.
By arbitrariness of $\varepsilon>0$, this entails that
\[
\frac{R^{\rm G}_p(t\,\Omega)-\delta}{t}<\frac{r_\delta}{t}\le R^{\rm G}_p(\Omega).
\]
Finally, by taking the limit as $\delta$ goes to $0$, we obtain
\[
\frac{R^{\rm G}_p(t\,\Omega)}{t}\le R^{\rm G}_p(\Omega).
\]
To prove the reverse inequality, we can proceed similarly. For every $\delta>0$, there exists $R^{\rm G}_p(\Omega)-\delta<r_\delta<R^{\rm G}_p(\Omega)$. Thus, for every $\varepsilon>0$ there exists a point $x_0\in\mathbb{R}^N$ such that
\[
\mathscr{C}_p(\overline{B_{r_\delta}(x_0)}\setminus \Omega)<\varepsilon.
\]
Observe that
\[
\overline{B_{r_\delta}(x_0)}\setminus\Omega=\frac{1}{t}\,t\,\left(\overline{B_{{r_\delta}}(x_0)}\setminus \Omega\right)=\frac{1}{t}\,\left(\overline{B_{t\,{r_\delta}}(t\,x_0)}\setminus(t\,\Omega)\right),
\]
thus this time we have
\[
\min\left\{\left(\frac{1}{t}\right)^{N-p},\left(\frac{1}{t}\right)^{N}\right\}\,\mathscr{C}_p\left(\overline{B_{t\,{r_\delta}}(t\,x_0)}\setminus(t\,\Omega)\right)<\,\varepsilon.,
\]
again by Lemma \ref{lm:nonno}.
By arbitrariness of $\varepsilon>0$, this entails that
\[
t\,(R^{\rm G}_p(\Omega)-\delta)<t\,r_\delta\le R^{\rm G}_p(t\,\Omega).
\]
By taking the limit as $\delta$ goes to $0$, we conclude.
\end{proof}
The capacitary inradius $R^{\rm G}_p$ admits the following equivalent definition, in terms of the relative $p-$capacity. The following is a particular case of \cite[Lemma 2.11]{Ga2}.
\begin{lemma}
\label{lm:equivalent}
Let $1\le p<\infty$. For every open set $\Omega\subseteq\mathbb{R}^N$ we have
\[
R^{\rm G}_{p}(\Omega):=\sup\Big\{r>0\, :\, \forall \varepsilon>0,\ \exists x_0\in\mathbb{R}^N\ \text{such that}\ \mathrm{cap}_p\left(\overline{B_r(x_0)}\setminus\Omega;B_{2r}(x_0)\right)< \varepsilon\Big\}.
\]
\end{lemma}
\begin{proof}
We suppose that $R^{\rm G}_{p}(\Omega)<+\infty$, in the case $R^{\rm G}_{p}(\Omega)=+\infty$ the argument below can be easily adapted.
For every $\delta>0$, by definition there exists $R^{\rm G}_p(\Omega)-\delta<r_\delta<R^{\rm G}_p(\Omega)$ such that for every $\varepsilon>0$ we have 
\[
\mathscr{C}_p(\overline{B_{r_\delta}(x_0)}\setminus\Omega)<\varepsilon,
\]
for some $x_0\in\mathbb{R}^N$. In particular, 
from Lemma \ref{lm:ammazya!} with $\Sigma=\overline{B_{r_\delta}(x_0)}\setminus\Omega$ and $E=B_{2r_\delta}(x_0)$, we get
\[
\mathrm{cap}_p\left(\overline{B_{r_\delta}(x_0)}\setminus\Omega;B_{2r_\delta}(x_0)\right)< \varepsilon\,2^{p-1}\,\max\left\{1,\frac{1}{r_\delta}\right\}.
\]
By arbitrariness of $\varepsilon>0$, this implies that
\[
\sup\Big\{r>0\, :\, \forall \varepsilon>0,\ \exists x_0\in\mathbb{R}^N\ \text{such that}\ \mathrm{cap}_p\left(\overline{B_r(x_0)}\setminus\Omega;B_{2r}(x_0)\right)< \varepsilon\Big\}\ge r_\delta.
\]
By recalling that $R^{\rm G}_p(\Omega)-\delta<r_\delta<R^{\rm G}_p(\Omega)$, the arbitrariness of $\delta$ implies that
\[
\sup\Big\{r>0\, :\, \forall \varepsilon>0,\ \exists x_0\in\mathbb{R}^N\ \text{such that}\ \mathrm{cap}_p\left(\overline{B_r(x_0)}\setminus\Omega;B_{2r}(x_0)\right)< \varepsilon\Big\}\ge R^{\rm G}_{p}(\Omega).
\]
The reverse inequality can be proved in a similar way, by using again Lemma \ref{lm:ammazya!}.
\end{proof}

About the comparison between $R_{p,\gamma}$ and $R^{\rm G}_p$, we have the following result, which is essentially contained in \cite{Ga2}, up to some technical adjustments in the proof.
\begin{proposition}
\label{prop:GvsUs}
Let $1\le p\le N$, then for every $\Omega\subseteq\mathbb{R}^N$ open set we have
\[
R^{\rm G}_p(\Omega)= \inf_{0<\gamma<1} R_{p,\gamma}(\Omega)=\lim_{\gamma\searrow 0} R_{p,\gamma}(\Omega).
\]
\end{proposition}
\begin{proof}
The fact that 
\[
\inf_{0<\gamma<1} R_{p,\gamma}(\Omega)=\lim_{\gamma\searrow 0} R_{p,\gamma}(\Omega),
\]
simply follows from the monotonicity of $\gamma\mapsto R_{p,\gamma}(\Omega)$. Thus, it is sufficient to prove the first identity. We then divide the proof in two parts.
\vskip.2cm\noindent
{\it Part 1}. Here we show that
\begin{equation}
\label{gallagherLB}
R_{p,\gamma}(\Omega)\ge R^{\rm G}_{p}(\Omega),\qquad \text{for every}\ \gamma\in(0,1),
\end{equation}
by proceeding as in the proof of \cite[Lemma 2.14]{Ga2}.
Let us fix $\gamma\in(0,1)$. We can assume that $R^{\rm G}_{p}(\Omega)<+\infty$, in the case $R^{\rm G}_{p}(\Omega)=+\infty$ the reader will immediately see how to adapt the argument below. By definition of supremum for every $\delta>0$ there exists $R^{\rm G}_{p}(\Omega)-\delta<r_\delta<R^{\rm G}_{p}(\Omega)$ with the following property: for every $\varepsilon>0$ there exists a point $x_0\in\mathbb{R}^N$ such that 
\[
\mathrm{cap}_p(\overline{B_{r_\delta}(x_0)}\setminus\Omega;B_{2{r_\delta}}(x_0))<\varepsilon.
\]
We also used the equivalent formulation of $R^{\rm G}_p(\Omega)$ contained in Lemma \ref{lm:equivalent}.
Thus, if we use the property above with the choice 
\[
\varepsilon=\gamma\,\mathrm{cap}_p(\overline{B_{r_\delta}};B_{2{r_\delta}}),
\] 
we get existence of a center $x_0$ such that
\[
\gamma\,\mathrm{cap}_p(\overline{B_{r_\delta}(x_0)};B_{2{r_\delta}}(x_0))>\mathrm{cap}_p(\overline{B_{r_\delta}(x_0)}\setminus\Omega;B_{2{r_\delta}}(x_0)).
\]
This shows that the radius ${r_\delta}$ is admissible for the definition of $R_{p,\gamma}(\Omega)$. Thus, we obtain
\[
R_{p,\gamma}(\Omega)\ge r_\delta>R^{\rm G}_{p}(\Omega)-\delta. 
\]
By arbitrariness of $\delta>0$, this in turn implies \eqref{gallagherLB}.
\vskip.2cm\noindent
{\it Part 2}. 
In order to prove that
\begin{equation}
\label{reverso}
\lim_{\gamma\nearrow 0} R_{p,\gamma}(\Omega)\le R^{\rm G}_p(\Omega),
\end{equation}
we need to distinguish two possibilities\footnote{We recall that the condition $R_{p,\gamma}(\Omega)<+\infty$ does not depend on the particular $\gamma\in(0,1)$, thanks to \eqref{capin}.} :
\begin{itemize}
\item[(i)] either $R_{p, \gamma}(\Omega) = +\infty$ for all $0 < \gamma < 1$;
\vskip.2cm
\item[(ii)] or $R_{p, \gamma}(\Omega) <+\infty$ for all $0 < \gamma < 1$.
\end{itemize} 
If alternative (i) occurs, we claim that we must have $R^{\rm G}_p(\Omega) = +\infty$, as well. This would establish \eqref{reverso}.
Indeed, let us fix a radius $r>0$, for every $\varepsilon>0$ we set 
\[
\gamma_\varepsilon=\min\left\{\frac{1}{2},\frac{\varepsilon}{2\,\mathrm{cap}_p(\overline{B_r}; B_{2r})}\right\}=\min\left\{\frac{1}{2},\frac{\varepsilon}{2\,r^{N-p}\,\mathrm{cap}_p(\overline{B_1}; B_{2})}\right\}.
\]
Since $R_{p,\gamma_\varepsilon}(\Omega)=+\infty$, there exists a point $x_0\in\mathbb{R}^N$ such that
\[
\mathrm{cap}_p(\overline{B_r(x_0)} \setminus \Omega; B_{2r}(x_0))\le \gamma_\varepsilon\,\mathrm{cap}_p(\overline{B_r(x_0)}; B_{2r}(x_0))<\varepsilon.
\]
According to Lemma \ref{lm:equivalent}, this shows that $r\le R^{\rm G}_p(\Omega)$. By arbitrariness of $r>0$, we thus obtain $R^{\rm G}_p(\Omega)=+\infty$, as claimed.
\par 
 We now assume alternative (ii): observe that we also have $R_{p}^{\rm G}(\Omega)<+\infty$ by {\it Part 1}. We need to show that for every $0<\delta\ll 1$, there exists $0<\gamma_\delta<1$ such that
\begin{equation}
\label{santita}
R_{p,\gamma_\delta}(\Omega)\le R^{\rm G}_p(\Omega)+\delta.
\end{equation}
This would be sufficient to get \eqref{reverso}. We preliminary observe that we can suppose that
\[
R^{\rm G}_p(\Omega)< R_{p,1/2}(\Omega).
\]
Otherwise, in light of {\it Part 1}, we would have $R^{\rm G}_p(\Omega)= R_{p,1/2}(\Omega)$ and thus \eqref{santita} would follow with $\gamma_\delta\equiv1/2$.
\par
We then choose any $0<\delta<\delta_0:=R_{p,1/2}(\Omega)-R^{\rm G}_p(\Omega)$ so that we still have 
\[
R^{\rm G}_p(\Omega)+\delta<R_{p,1/2}(\Omega).
\]
Observe that since $R^{\rm G}_p(\Omega)+\delta>R_p^{\rm G}(\Omega)$, there exists an $\varepsilon_\delta>0$ such that
\[
\mathrm{cap}_p\left(\overline{B_{R^{\rm G}_p(\Omega)+\delta}(x_0)}\setminus\Omega;B_{2(R^{\rm G}_p(\Omega)+\delta)}(x_0)\right)
\ge \varepsilon_\delta,\qquad \text{for every}\ x_0\in\mathbb{R}^N,
\]
thanks to the equivalent formulation of Lemma \ref{lm:equivalent}.
Thus, from Lemma \ref{lemma:3} we get for every $r>R^{\rm G}_p(\Omega)+\delta$
\[
\mathrm{cap}_p(\overline{B_{r}(x_0)}\setminus\Omega;B_{2r}(x_0))\ge \varepsilon_\delta,\qquad \text{for every}\ x_0\in\mathbb{R}^N,
\]
In particular, for every radius $r$ such that 
\[
R^{\rm G}_p(\Omega)+\delta<r\le R_{p,1/2}(\Omega),
\] 
we also have
\[
\varepsilon_\delta\,\frac{\mathrm{cap}_p(\overline{B_r(x_0)};B_{2r}(x_0))}{\Big(R_{p,1/2}(\Omega)\Big)^{N-p}\,\mathrm{cap}_p(\overline{B_1};B_2)}\le \mathrm{cap}_p(\overline{B_r(x_0)}\setminus \Omega;B_{2r}(x_0)),
\]
thanks to the scaling properties of the relative $p-$capacity and the restriction $r\le R_{p,1/2}(\Omega)$.
In particular, if we define
\[
\gamma_\delta=\min\left\{\frac{1}{4},\ \frac{\varepsilon_\delta}{2}\, \frac{1}{\Big(R_{p,1/2}(\Omega)\Big)^{N-p}\,\mathrm{cap}_p(\overline{B_1};B_2)}\right\},
\]
from the previous estimate we get
\[
\gamma_\delta\,\mathrm{cap}_p(\overline{B_r(x_0)};B_{2r}(x_0))<\mathrm{cap}_p(\overline{B_r(x_0)}\setminus \Omega;B_{2r}(x_0)).
\]
This holds for every $R^{\rm G}_p(\Omega)+\delta<r\le R_{p,1/2}(\Omega)$ and every $x_0\in\mathbb{R}^N$. Thus, we have two possibilities: 
\[
\text{either}\ R_{p,\gamma_\delta}(\Omega)\le R^{\rm G}_p(\Omega)+\delta\qquad \text{or}\qquad R_{p,\gamma_\delta}(\Omega)> R_{p,1/2}(\Omega).
\]
Since $\gamma_\delta\le 1/4$ by construction and $\gamma\mapsto R_{p,\gamma}(\Omega)$ is non-decreasing, the second alternative can not occur. Thus, we finally get \eqref{santita}.
\end{proof}
We conclude this section by noticing that while we have (see \cite[Theorem 1.3]{Ga2})
\[
\lambda_p(\Omega)\le \lambda_p(B_1)\,\left(\frac{1}{R^{\rm G}_p(\Omega)}\right)^p,
\]
it is {\it not} possible to have a lower bound of the type
 	\begin{equation}
	\label{Gnot}
	C_{N,p}\, \left(\frac{1}{R^{\rm G}_p(\Omega)}\right)^p \le \lambda_p(\Omega), \quad \text{for every open set}\ \Omega \subseteq \mathbb{R}^N.
	\end{equation}
	This is the content of the next example.
\begin{example}
\label{exa:0fesso}
Let $0<\delta<1/4$ be fixed, we introduce the periodically perforated set
	\[
	\Omega_\delta= \mathbb{R}^N \setminus \bigcup_{\mathbf{i}\in\mathbb{Z}^N} \overline{B_\delta(\mathbf{i})},
	\] 
	as in \cite[Example A.1]{BozBra2}.
	Take $r > \sqrt{N}/2+ 1/4$. Let $x_0 \in \mathbb{R}^N$ and let ${\bf i}_0\in \mathbb{Z}^N$ be such that 
	\[
	|x_0 - {\bf i}_0| = {\rm dist}\left(x_0, \mathbb{R}^N \setminus \mathbb{Z}^N\right). 
	\]
	For every $y \in B_{\delta}({\bf i}_0)$, by using the triangle inequality, we have 
	\[
	|y - x_0| \le |y - {\bf i}_0| + |{\bf i}_0 - x_0| < \delta + \frac{\sqrt{N}}{2} < \frac{1}{4} + \frac{\sqrt{N}}{2} < r,
	\]
	that is 
	\[
	B_{\delta}({\bf i}_0) \subseteq B_r(x_0).
	\]
	This entails that 
	\begin{equation} \label{inclusion}
		|\overline{B_r(x_0)} \setminus \Omega_\delta| \ge |B_\delta(x_0)| = \omega_N\,\delta^N, \qquad \mbox{ for every } x_0 \in \mathbb{R}^N. 
	\end{equation}
	Then, by using Lemma \ref{lm:ammazya!} in combination with Lemma \ref{lemma:2}, we get
	\[
		\begin{split}
			\mathscr{C}_p(\overline{B_{r}(x_0)} \setminus \Omega_\delta) &\ge 2^{1-p}\,\frac{r^p}{r^p+1}\, \mathrm{cap}_p\left(\overline{B_r(x_0)} \setminus \Omega_\delta; B_{2r}(x_0)\right) \\
			&\ge2^{1-p}\,\frac{r^p}{r^p+1}\,\lambda_p(B_{2r}(x_0))\,|\overline{B_r(x_0)} \setminus \Omega_\delta|.\\
		\end{split}
\]
By using that $\lambda_p(B_{2r}(x_0))=(2\,r)^{-p}\,\lambda_p(B_1)$ and \eqref{inclusion}, we thus get
\[
\mathscr{C}_p(\overline{B_{r}(x_0)} \setminus \Omega_\delta) \ge \varepsilon_0,\qquad \text{where}\ \varepsilon_0=\frac{2}{4^p}\, \frac{\omega_N}{r^p+1}\,\lambda_p(B_1)\,\delta^N.
\] 
By arbitrariness of $x_0\in\mathbb{R}^N$, this shows that we must have $r > R^{\rm G}_p(\Omega_\delta)$. In turn, by the arbitrariness of $r$ we also have
	\[
	R^{\rm G}_p(\Omega_\delta) \le \frac{\sqrt{N}}{2} + \frac{1}{4}, \qquad \text{for every}\ 0 < \delta < \frac{1}{4}. 
	\]
	On the other hand, by using the same computations of \cite[Example A.1]{BozBra2} we have that 
	\[
	\lim_{\delta \searrow 0} \lambda_p(\Omega_{\delta}) = 0. 
	\]
	Thus, an estimate like \eqref{Gnot} can not be true.
\end{example}

\section{Capacitary inradius and measure density condition}	
\label{sec:5}

\subsection{A two-sided estimate}
In this section, we want to compare our capacitary inradius with the classical inradius, under a simple measure density condition on the open sets. Namely, we will consider open sets $\Omega\subseteq\mathbb{R}^N$ such that the following {\it measure density index}
\[
\theta^*_{\Omega,r_0}(t) := \inf\left\{\left(\frac{r_0}{r}\right)^t\,\dfrac{|B_r(x)\setminus \Omega|}{|B_r(x)|}\, : x\in \partial \Omega,\ 0 < r \le r_0\right\},
\]
is positive, for some $r_0>0$ and $t\ge 0$. We notice that by definition we have
\[
\theta^*_{\Omega,r_0}(t)\le \dfrac{|B_{r_0}(x)\setminus \Omega|}{|B_{r_0}(x)|}\le 1.
\]
\par
The next simple result gives a quantitative equivalence of this index with a slightly different one, which will be more practical.
\begin{proposition}
\label{prop:index}
	Let $\Omega \subsetneq \mathbb{R}^N$ be an open set and let $t\ge 0$, $r_0 > 0$. We set
	\begin{equation}
	\label{altind}
	\theta_{\Omega,r_0}(t) = \inf\left\{\left(\frac{r_0}{r}\right)^t\,\dfrac{|B_r(x)\setminus \Omega|}{|B_r(x)|}\, : x\in \mathbb{R}^N\setminus\Omega,\ 0 < r \le r_0\right\}.
	\end{equation}
Then we have 
	\begin{equation*}
		\theta_{\Omega, r_0}(t) \le \theta^*_{\Omega,r_0}(t) \le 2^{N+t}\,\theta_{\Omega, r_0}(t).
	\end{equation*}
\end{proposition}
\begin{proof}
This is essentially a particular case of \cite[Theorem 3.1]{GL}. We reproduce the proof to keep track of the relevant constant.
\par
	The leftmost inequality follows from the fact that $\partial \Omega \subseteq \mathbb{R}^N \setminus \Omega$. In order to prove the rightmost inequality, we can assume that 
	\begin{equation} \label{hp:meas-density-boundary}
		\theta^*_{\Omega,r_0} > 0,
	\end{equation}
	otherwise there is nothing prove.
 We fix a point $x \in \mathbb{R}^N \setminus \Omega$ and a radius $0 < r \le r_0$ be given. We set 
 \[
 d_\Omega(x)=\min_{y\in\partial\Omega} |x-y|,
 \]
 and distinguish two cases, for $0<\varepsilon<1$.
	\vskip.2cm\noindent
	{\it Case $d_{\Omega}(x) \ge(1-\varepsilon)\,r$.} By observing that $B_{(1-\varepsilon)\,r}(x) \subseteq \mathbb{R}^N \setminus \Omega$,
 	we have 
	\begin{equation} \label{ineq:case1-theta}
	\begin{split}
		|B_{r}(x) \setminus \Omega| \ge |B_{(1-\varepsilon)\,r}(x) \setminus \Omega| = |B_{(1-\varepsilon)\,r}(x)| &= \omega_N\,(1-\varepsilon)^N\,r^N\\
		  \ge \theta^*_{\Omega,r_0}\,\left(\frac{r}{r_0}\right)^t\, \omega_N\,(1-\varepsilon)^N\,r^N,
		\end{split}
	\end{equation}
	where in the last inequality we used the trivial fact that $\theta^*_{\Omega,r_0} \le 1$ and $r\le r_0$.
	\vskip.2cm\noindent 
	{\it Case $d_{\Omega}(x)< (1-\varepsilon)\,r$.} In this case, it must result $B_{(1-\varepsilon)\,r}(x) \cap \partial \Omega \neq \emptyset$.
Given $y \in B_{(1-\varepsilon)\,r}(x) \cap \partial \Omega$, we set $\delta := r - |x-y|.$ In particular, by construction we have 
	\[
	\varepsilon\, r < \delta \le r  \qquad \mbox{ and } \qquad B_\delta(y) \subseteq B_{r}(x).
	\]
	By using \eqref{hp:meas-density-boundary}, we then infer that 
	\begin{equation} \label{ineq:case2-theta}
		\begin{split}
			|B_{r}(x) \setminus \Omega| \ge |B_\delta(y) \setminus \Omega| \ge \theta^*_{\Omega,r_0}\,\left(\frac{\delta}{r_0}\right)^t \omega_N\,\delta^N > \theta^*_{\Omega,r_0}\,\left(\frac{r}{r_0}\right)^t \omega_N\,\varepsilon^{N+t}\,r^N . 
		\end{split}
	\end{equation}
	If we now combine \eqref{ineq:case1-theta} and \eqref{ineq:case2-theta}, we eventually obtain 
	\[
	\left(\frac{r_0}{r}\right)^t\,\dfrac{|B_{r}(x) \setminus \Omega|}{|B_{r}(x)|} \ge \min\left\{(1-\varepsilon)^N, \varepsilon^{N+t}\right\}\,\theta^*_{\Omega,r_0}. 
	\]
	By arbitrariness of both $x \in \mathbb{R}^N \setminus \Omega$ and $0<r\le r_0$, 
	we get 
\[
\theta_{\Omega, r_0}\ge \min\left\{(1-\varepsilon)^N, \varepsilon^{N+t}\right\}\,\theta^*_{\Omega,r_0}.
\]
The choice $\varepsilon=1/2$ leads to the desired conclusion.
\end{proof}
We now come to the main result of this paper.
\begin{theorem} 
\label{prop:meas-density-ball}
	Let $1 \le p \le N$ and let $\Omega \subsetneq \mathbb{R}^N$ be an open set such that $\theta^*_{\Omega, r_0}(t)>0$,
	for some $r_0 > 0$ and $t\ge 0$. Then there exists an explicit parameter $0 < \gamma_0 = \gamma_0(N,p,\theta^*_{\Omega, r_0}(t),t,r_0/r_\Omega) < 1$ such that if we set
\[
\mathcal{C}=\left\{\begin{array}{rl}
6\,\sqrt{N},& \text{if}\ 0<\gamma<\gamma_0,\\ 
&\\
6\,\sqrt{N}\,\left(\dfrac{2\,C_{N,p,\gamma}}{\gamma_0\,\sigma_{N,p}}\right)^\frac{1}{p},& \text{if}\ \gamma_0\le \gamma<1,
\end{array}
\right.
\]
with $C_{N,p,\gamma}$ and $\sigma_{N,p}$ as in \eqref{capin},	
then we have
	\begin{equation}
	\label{twosided}
		r_\Omega\le R_{p,\gamma}(\Omega) \le \mathcal{C}\,r_{\Omega}, \qquad \text{for every}\ 0<\gamma <1.
	\end{equation}
\end{theorem}
\begin{proof}
The leftmost inequality in \eqref{twosided} follows from the definition of capacitary inradius, we focus on the other estimate.
Without loss of generality we can assume that $r_{\Omega} < +\infty$.  We divide the proof in two parts: we first construct $\gamma_0$ and show that the estimate is true in the range $0<\gamma<\gamma_0$; then we show how this is sufficient to get the claimed estimate for $\gamma\ge \gamma_0$, as well.
\vskip.2cm\noindent
{\it Part 1: case $\gamma<\gamma_0$}.
We set 
	\begin{equation} \label{defi:parameters}
		\alpha := \frac{r_0}{r_\Omega} \qquad \mbox{ and } \qquad A := 2\,\sqrt{N}\,\left(2 + \min\{1, \alpha\}\right),
	\end{equation}
	in particular $A \le 6\,\sqrt{N}$.
	We take an open ball $B_r(x_0)$ having radius $r > A\, r_{\Omega}$. 
	We set\footnote{For every $\delta\in\mathbb{R}$, we denote by $\big\lfloor\delta\big\rfloor$ its {\it integer part}, defined by 
\[
\big\lfloor\delta\big\rfloor=\max\Big\{n\in\mathbb{Z}\, :\, \delta\ge n\Big\}.
\]
} 
	\begin{equation} \label{defi:erre-segnato}
	\ell=r_\Omega\,(2+\min\{1,\alpha\})\qquad \text{and}\qquad m=\left\lfloor\frac{r/\sqrt{N}}{\ell}\right\rfloor,
	\end{equation}
then 
\[
Q_{m\,\ell}(x_0)\subseteq Q_{r/\sqrt{N}}(x_0) \subseteq B_r(x_0).
\]
Moreover, the cube $Q_{m\,\ell}(x_0)$ can be tiled (up to a Lebesgue negligible set) with $m^N$ disjoint cubes having side length $2\,\ell$, let us call $\{Q_\ell(\mathbf{p}_i)\}_{i=1}^{m^N}$ these cubes.	By construction, the cube $Q_\ell(\mathbf{p}_i)$ contains the ball $B_{2\,r_\Omega}(\mathbf{p}_i)$. Thanks to the definition of inradius, there exists in particular a point $x_i\in B_{2\,r_\Omega}(\mathbf{p}_i)\setminus\Omega$. By recalling the definition of $\ell$, we have that 
\[
B_{\ell-2\,r_\Omega}(x_i)\subseteq Q_\ell(\mathbf{p}_i).
\]
\begin{figure}
\includegraphics[scale=.3]{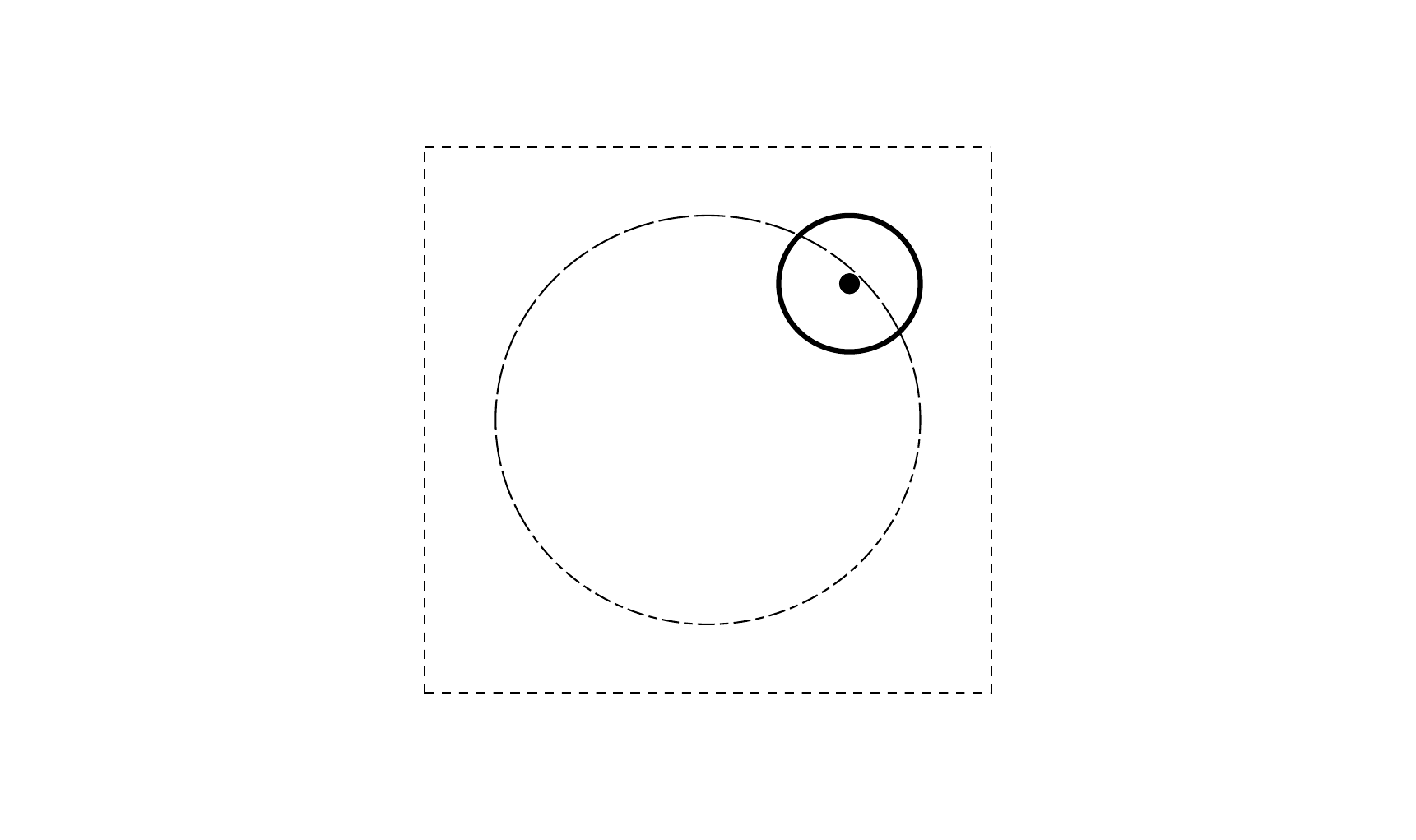}
\caption{A cube $Q_\ell(\mathbf{p}_i)$, containing a concentric ball with radius $2\,r_\Omega$. The black dot indicates a point outside $\Omega$. By construction, at this point we can center a ball which is feasible for the measure density index $\theta_{\Omega,r_0}(t)$.}
\label{fig:1}
\end{figure}
Observe that $\ell-2\,r_\Omega=r_\Omega\,\min\{1,\alpha\}\le r_0$. Thus, in particular we have that each ball $B_{\ell-2\,r_\Omega}(x_i)$ is feasible for the definition of the amended measure density index $\theta_{\Omega,r_0}(t)$ defined by \eqref{altind} (see Figure \ref{fig:1}). Accordingly, we get
\[
\begin{split}
|B_{\ell-2\,r_\Omega}(x_i)\setminus\Omega|&\ge \theta_{\Omega,r_0}(t)\,\omega_N\,(\ell-2\,r_\Omega)^N\,\left(\frac{\ell-2\,r_\Omega}{r_0}\right)^t\\
&=\theta_{\Omega,r_0}(t)\,\omega_N\,\min\{1,\alpha^{N+t}\}\,r_\Omega^{N}\,\left(\frac{r_\Omega}{r_0}\right)^t\\
&=\theta_{\Omega,r_0}(t)\,\omega_N\,\min\{\alpha^{-t},\alpha^{N}\}\,r_\Omega^{N},
\end{split}
\]
By keeping into account all the contributions of these balls, we can thus estimate
\begin{equation}
\label{palline}
\begin{split}
|B_r(x_0)\setminus \Omega|\ge \sum_{i=1}^{m^N} |Q_\ell(\mathbf{p}_i)\setminus\Omega|&\ge \sum_{i=1}^{m^N}|B_{\ell-2\,r_\Omega}(x_i)\setminus\Omega|\\
&\ge m^N\,\theta_{\Omega,r_0}(t)\,\omega_N\,\min\{\alpha^{-t},\alpha^{N}\}\,r_\Omega^{N}.
\end{split}
\end{equation}
We claim that the latter can be further estimated from below by a term of the type $c\,r^N$. Indeed,  by recalling that we are assuming $r>A\,r_\Omega$ and using both \eqref{defi:parameters} and \eqref{defi:erre-segnato}, the number of cubes $m$ is such that
\[
	m=	\left\lfloor \frac{r/\sqrt{N}}{\ell} \right\rfloor \ge \frac{r/\sqrt{N}}{\ell} - 1 > \frac{r/\sqrt{N}}{\ell} - \frac{\sqrt{N}}{A}\frac{r/\sqrt{N}}{r_{\Omega}} = \frac{r}{r_\Omega}\,\frac{1}{2\,\sqrt{N}}\frac{1}{2+\min\{1, \alpha\}} .
\]
We spend this information in \eqref{palline}, so to obtain
\[
|B_r(x_0)\setminus \Omega|\ge \left(\frac{1}{2\,\sqrt{N}}\,\frac{\min\{\alpha^{-t/N},\alpha\}}{2+\min\{1, \alpha\}}\right)^N\,\theta_{\Omega,r_0}(t)\,\omega_N\,r^N.
\]
By further using Proposition \ref{prop:index}, we can replace $\theta_{\Omega,r_0}(t)$ with $\theta_{\Omega,r_0}^*(t)$, up to modify the constant. Namely, we can infer	
\begin{equation} \label{eqn:fat-vol}
		|\overline{B_r(x_0)} \setminus \Omega| 
			\ge \left(\frac{1}{4\cdot 2^{t/N}\,\sqrt{N}}\,\frac{\min\{\alpha^{-t/N},\alpha\}}{2+\min\{1, \alpha\}}\right)^N\,\theta^*_{\Omega, r_0}(t)\,\omega_N\,r^N.
	\end{equation}
On the other hand, by using Lemma \ref{lemma:2} with $K = \overline{B_r(x_0)} \setminus\Omega$ and $E=B_{2r}(x_0)$, the measure of $\overline{B_r(x_0)} \setminus \Omega$ bounds from below its relative $p-$capacity. More precisely, we have for $p<N$
	\[
	\mathrm{cap}_p\left(\overline{B_r(x_0)} \setminus \Omega; B_{2r}(x_0)\right)\ge \frac{\lambda_p(B_2)}{r^p}\,|\overline{B_r(x_0)} \setminus \Omega|.
	\]
	On account of the lower bound \eqref{eqn:fat-vol}, this yields
\[
\begin{split}
\mathrm{cap}_p&\left(\overline{B_r(x_0)} \setminus \Omega; B_{2r}(x_0)\right)\\
&\ge \lambda_p(B_2)\,|B_1|\,\left(\frac{1}{4\cdot 2^{t/N}\,\sqrt{N}}\,\frac{\min\{\alpha^{-t/N}, \alpha\}}{2+\min\{1, \alpha\}}\right)^{N}\,\theta^*_{\Omega, r_0}(t)\,r^{N-p}.
\end{split}
\]		
On the other hand, if $B_r(x_0)$ would be $(p,\gamma)-$negliglible, we would have
\[
			\gamma\,r^{N-p}\,\mathrm{cap}_p\left(\overline{B}_1; B_2\right)\ge \mathrm{cap}_p\left(\overline{B_r(x_0)} \setminus \Omega; B_{2r}(x_0)\right).
\]
	By joining the last two estimates and canceling out the term $r^{N-p}$ on both sides, we would eventually obtain that it must result\footnote{Observe that $\gamma_0<1$, since by Lemma \ref{lemma:2} with $E=B_2$ and $\Sigma=\overline{B_1}$ we have
	\[
	\frac{\lambda_p(B_2)\,|B_1|}{\mathrm{cap}_p\left(\overline{B}_1; B_2\right)}\le 1.
	\]}
		\begin{equation}
		\label{gamma0}
	\gamma\ge 	\frac{\lambda_p(B_2)\,|B_1|}{\mathrm{cap}_p\left(\overline{B}_1; B_2\right)}\,\left(\frac{1}{4\cdot 2^{t/N}\,\sqrt{N}}\,\frac{\min\{\alpha^{-t/N}, \alpha\}}{2+\min\{1, \alpha\}}\right)^N\,\theta^*_{\Omega, r_0}(t)=:\gamma_0.
	\end{equation}
		This shows that every ball with radius $r > A\,r_{\Omega}$ can not be $(p,\gamma)-$negligible, for $\gamma < \gamma_0$. On account of the definition \eqref{cap_inradius} of capacitary inradius, we obtain 
	\[
	R_{p,\gamma}(\Omega) \le A\,r_{\Omega},\qquad \text{for every}\ 0<\gamma < \gamma_0.
	\]
As already observed, we have in particular $A \le 6\,\sqrt{N}$ and thus our claim follows.
\vskip.2cm\noindent
{\it Part 2: case $\gamma\ge \gamma_0$}. We can exploit the two-sided estimate \eqref{capin}. Indeed, by using it with $\gamma\ge \gamma_0$ on the right-hand side and with $\gamma=\gamma_0/2$ on the left-hand side, we get in particular
\[
\sigma_{N,p}\,\frac{\gamma_0}{2}\,\left(\frac{1}{R_{p,\gamma_0/2}(\Omega)}\right)^p\le \lambda_p(\Omega)\le C_{N,p,\gamma}\,\left(\frac{1}{R_{p,\gamma}(\Omega)}\right)^p.
\]
Thus, we also obtain
\[
R_{p,\gamma}(\Omega)\le \left(\frac{2\,C_{N,p,\gamma}}{\sigma_{N,p}\,\gamma_0}\right)^\frac{1}{p}\, R_{p,\gamma_0/2}(\Omega).
\]
By using the result of the first part, we now conclude.
\end{proof} 
\begin{remark}[The case $p>N$]
This case is simpler and one can prove that $R_{p,\gamma}(\Omega)$ and $r_\Omega$ are equivalent {\it for every open set}. This is clearly due to the fact that for $p>N$ the only set with zero $p-$capacity is the empty set. We refer to \cite[Section 7]{BozBra2} for the discussion of this case.
\end{remark}

\begin{remark}
From the explicit expression \eqref{gamma0} of the parameter $\gamma_0$, we see that we have
\[
\gamma_0\to 0\qquad \text{if}\quad \theta^*_{\Omega,r_0}(t)\to 0,
\]
and also
\[
\gamma_0\to 0\qquad \text{if either}\quad \alpha=\frac{r_0}{r_\Omega}\to +\infty\quad \text{or}\quad \alpha=\frac{r_0}{r_\Omega}\to 0.
\]
The case $t=0$ is special: in this case it is easily seen from \eqref{gamma0} that $\alpha^{-t/N}=1$ and thus $\gamma_0$ stays bounded away from $0$, even if the ratio $r_0/r_\Omega$ diverges. We point out that for $t=0$ the measure density condition becomes
\[
\theta^*_{\Omega,r_0}(0) := \inf\left\{\dfrac{|B_r(x)\setminus \Omega|}{|B_r(x)|}\, : x\in \partial \Omega,\ 0 < r \le r_0\right\}>0,
\]
which sounds like a uniform cone condition (but it is actually weaker than this, see \cite[Theorem 4.1]{GL}).
\end{remark}
On the other hand, from the previous theorem we have that 
\[
r_\Omega \le \lim_{\gamma\searrow 0} R_{p,\gamma}(\Omega)\le 6\,\sqrt{N}\,r_\Omega,
\]
for every open set with a positive measure density index $\theta^*_{\Omega,r_0}(t)$. In particular, for the Gallagher capacitary inradius, from Proposition \ref{prop:GvsUs}  we get the following
\begin{corollary}
	Let $1 \le p \le N$ and let $\Omega \subsetneq \mathbb{R}^N$ be an open set such that $\theta^*_{\Omega,r_0}(t)>0$,
	for some $r_0 > 0$ and $t\ge 0$. Then
	\[
	R^{\rm G}_p(\Omega)\le 6\,\sqrt{N}\,r_\Omega.
	\]
\end{corollary}

\subsection{Some consequences}
We now highlight some interesting consequences which can be drawn from Theorem \ref{prop:meas-density-ball}. We start with the following two-sided estimate on the sharp Poincar\'e--Sobolev constants. 
\begin{corollary}
\label{coro:kroffo}
Let $1\le p\le N$ and let $q\ge p$ be such that
\[
	\begin{cases}
		q < p^{*}, \quad &\text{if}\ 1 \leq p < N,\\
		q < \infty, \quad &\text{if}\ p=N.\\
		\end{cases}
\]
Let $\Omega \subseteq \mathbb{R}^N$ be an open set such that $\theta^*_{\Omega,r_0}(t)>0$ for some $r_0>0$ and $t\ge 0$. Then we have
\[
\lambda_{p,q}(\Omega)>0\qquad\Longleftrightarrow \qquad r_\Omega<+\infty.
\]
Moreover, if we have
\[
t\le t_0,\qquad\theta^*_{\Omega,r_0}(t)\ge \theta>0\qquad \text{and}\qquad 0<\kappa_1\le \frac{r_0}{r_\Omega}\le \kappa_2,
\] 
then  there holds
	\begin{equation}
	\label{twosided_inradius}
\frac{c}{r_\Omega^{p-N+N\,\frac{p}{q}}}\le\lambda_{p,q}(\Omega)\le \frac{\lambda_{p,q}(B_1)}{r_\Omega^{p-N+N\,\frac{p}{q}}},
	\end{equation}
for some $c=c(N,p,q,t_0,\theta,\kappa_1,\kappa_2)>0$. Finally, if $t_0=0$ the previous constant can be taken independent of $\kappa_2$.
\end{corollary}
\begin{proof}
We directly prove the two-sided estimate \eqref{twosided_inradius}. The rightmost inequality trivially follows from the monotonicity of $\lambda_{p,q}$ with respect to the set inclusion and its scaling properties.
\par
For the leftmost one, we use that
\[
\gamma\,\sigma_{N, p, q} \left(\frac{1}{R_{p,\gamma}(\Omega)}\right)^{p-N+N\frac{p}{q}} \leq \lambda_{p,q}(\Omega),\qquad \text{for every}\ 0<\gamma<1,
\]
thanks to \cite[Main Theorem \& Theorem 6.1]{BozBra}.
We choose $\gamma=\gamma_0/2$, with $\gamma_0$ as in Theorem \ref{prop:meas-density-ball}, and use \eqref{twosided} to estimate $R_{p,\gamma_0/2}(\Omega)$ from above with $r_\Omega$. 
This gives 
\[
\frac{\gamma_0\,\sigma_{N, p, q} }{2}\,\left(\frac{1}{6\,\sqrt{N}}\right)^{p-N+N\frac{p}{q}}\,\frac{1}{{r_\Omega^{p-N+N\,\frac{p}{q}}}} \leq \lambda_{p,q}(\Omega).
\]
By recalling the explicit expression \eqref{gamma0} of $\gamma_0$, we get the desired estimate with a constant $c$ only depending on the claimed quantities.
\end{proof}
\begin{remark}
We remark that the previous result extends \cite[Proposition 2.1]{Sou}. Apart for admitting a general subcritical exponent $q\ge p$, we provide a quantitative lower bound on the relevant sharp Poincar\'e-Sobolev constant. Moreover, we notice that our measure density condition is much weaker than the cone condition required in \cite{Sou} (see Lemma \ref{lm:funnel} below).
\end{remark}
\begin{remark}
In the case $p>N$, the previous result holds for {\it every} open set and it has been obtained by Maz'ya, see \cite[Theorem 11.4.1]{Maz85}. For different proofs, see also
\cite[Theorem 1.3]{BozBra}, \cite[Corollary 5.9]{BraPriZag2}, \cite[Theorem 1.4.1]{Po1} and \cite[Theorem 1.1]{Vit}.
\end{remark}
In turn, from Corollary \ref{coro:kroffo} we can obtain the equivalence between different sharp Poincar\'e constants. For example, if for $p=2$ we use the following distinguished notation 
\[
\lambda(\Omega)=\inf_{\varphi\in C^\infty_0(\Omega)} \left\{\int_\Omega |\nabla \varphi|^2\,dx\, :\, \|\varphi\|_{L^2(\Omega)}=1\right\},
\]
and we recall 
the definition of Cheeger's constant
\[
h(\Omega)=\inf\left\{\frac{\mathcal{H}^{N-1}(\partial E)}{|E|}\, :\, E\Subset \Omega\ \text{open set with smooth boundary}\right\},
\]
we can get at first the Buser--type inequality announced in the Introduction.
\begin{corollary}[Buser--type inequality]
\label{coro:buser}
Let $\Omega \subseteq \mathbb{R}^N$ be an open set with $r_\Omega<+\infty$. Let us suppose that for some $r_0>0$ we have
\[
t\le t_0,\qquad \theta^*_{\Omega,r_0}(t)\ge \theta>0\qquad \text{and}\qquad 0<\kappa_1\le \frac{r_0}{r_\Omega}\le \kappa_2.	
\] 
Then we have
	\[
	\lambda(\Omega) \le C\,\Big(h(\Omega)\Big)^2,
	\]
for some $C=C(N,t_0,\theta,\kappa_1,\kappa_2)>0$. Finally, if $t_0=0$ the previous constant can be taken independent of $\kappa_2$.
\end{corollary}
\begin{proof}
We can assume that $r_{\Omega} < +\infty$, otherwise the claimed inequality would be trivial. We recall that $\lambda_{1,1}(\Omega)=\lambda_1(\Omega)=h(\Omega)$ (see \cite[Theorem 2.1.3]{Maz} for this fact). Thus, from Corollary \ref{coro:kroffo} with $p=q=1$, we get
\[
h(\Omega)\ge \frac{c}{r_\Omega},
\]
with the constant $c$ depending only on $N,t_0,\theta,\kappa_1$ and $\kappa_2$. By combining this lower bound with  the classical estimate	
\[
\lambda(\Omega)\le \frac{\lambda(B_1)}{r_\Omega^2},
\] 
recalled in the Introduction, we obtain the claimed inequality with
\[
C=\frac{\lambda(B_1)}{c^2}.
\]
This concludes the proof.
\end{proof}
More generally, we recall that we have the following Cheeger--type inequality
\[
\left(\frac{p}{q}\right)^q\,\Big(\lambda_p(\Omega)\Big)^\frac{q}{p}\le \lambda_q(\Omega),\qquad \text{for}\ q>p,
\]
which generalizes \eqref{ciga}, see for example \cite[Proposition 2.2]{BBP}.
With a similar argument as above, we may ``reverse'' this estimate and infer the following equivalence between sharp Poincar\'e constants. 
\begin{corollary}[Generalized Buser--type inequality]
	Let $\Omega \subseteq \mathbb{R}^N$ be an open set with $r_\Omega<+\infty$, satisfying the measure density condition $\theta^*_{\Omega,r_0}(t)>0$, for some $r_0>0$ and $t\ge 0$. Then for every $1\le p<q\le N$ we have
	\[
	\lambda_p(\Omega)>0\qquad \Longleftrightarrow \qquad \lambda_q(\Omega)>0.
	\]
Moreover, if we have
\[
t\le t_0,\qquad\theta^*_{\Omega,r_0}(t)\ge \theta>0\qquad \text{and}\qquad 0<\kappa_1\le \frac{r_0}{r_\Omega}\le \kappa_2,
\] 
then
\[
\lambda_q(\Omega)\le C\,\Big(\lambda_p(\Omega)\Big)^\frac{q}{p},
\]
for some $C=C(N,p,q,t_0,\theta,\kappa_1, \kappa_2)>0$. Finally, if $t_0=0$ the previous constant can be taken independent of $\kappa_2$.
\end{corollary}	

\subsection{Example: uniform exterior funnel condition}
In this final section, we give a class of open sets which satisfy the measure density condition previously exposed.
\begin{definition}
We fix an exponent $0<\beta\le 1$, two positive real numbers $\delta,h$ and a direction $\omega\in\mathbb{S}^{N-1}$. We call {\it $\beta-$funnel of axis $\omega$ with height $h$ and opening $\delta$} the following set
\[
\mathscr{F}_\beta(\omega;\delta,h)=\{x\in\mathbb{R}^N\,:\,\delta\,|x-\langle x,\omega\rangle\,\omega|^\beta\le \langle x,\omega\rangle\le h \}.
\]
We then say that an open set $\Omega\subsetneq\mathbb{R}^N$ satisfies a {\it uniform exterior $\beta-$funnel condition with height $h_0$ and opening $\delta$} if for every $x\in\partial\Omega$ there exists a direction $\omega_x\in\mathbb{S}^{N-1}$ with the property that
\[
x+\mathscr{F}_\beta(\omega_x;\delta,h_0)\subseteq \mathbb{R}^N\setminus\Omega.
\]
\end{definition}
We have the following technical result, which illustrates concrete cases of applicability of the results of this section.
\begin{lemma}
\label{lm:funnel}
Let $\Omega\subsetneq\mathbb{R}^N$ be an open set satisfying a uniform exterior $\beta-$funnel condition with height $h_0$ and opening $\delta$. Then for 
\begin{equation}
\label{raggiobeta}
r_0=\sqrt{h_0^2+\left(\frac{h_0}{\delta}\right)^\frac{2}{\beta}}\qquad \text{and}\qquad t=(N-1)\,\left(\frac{1}{\beta}-1\right),
\end{equation}
we have
\begin{equation}
\label{lbindex}
\theta^*_{\Omega,r_0}(t)\ge c_{N,\beta}\,\min\left\{\left(\frac{h_0^{1-\beta}}{\delta}\right)^{\frac{1}{\beta}\,\frac{N-1}{\beta}},\,\left(\frac{\delta}{h_0^{1-\beta}}\right)^\frac{1}{\beta}\right\},
\end{equation}
for an explicit constant $c_{N,\beta}$.
\end{lemma}
\begin{proof}
For every $h\in(0,h_0]$, we define the radius
\[
r=\psi(h)=\sqrt{h^2+\left(\frac{h}{\delta}\right)^\frac{2}{\beta}}.
\] 
\begin{figure}
\includegraphics[scale=.3]{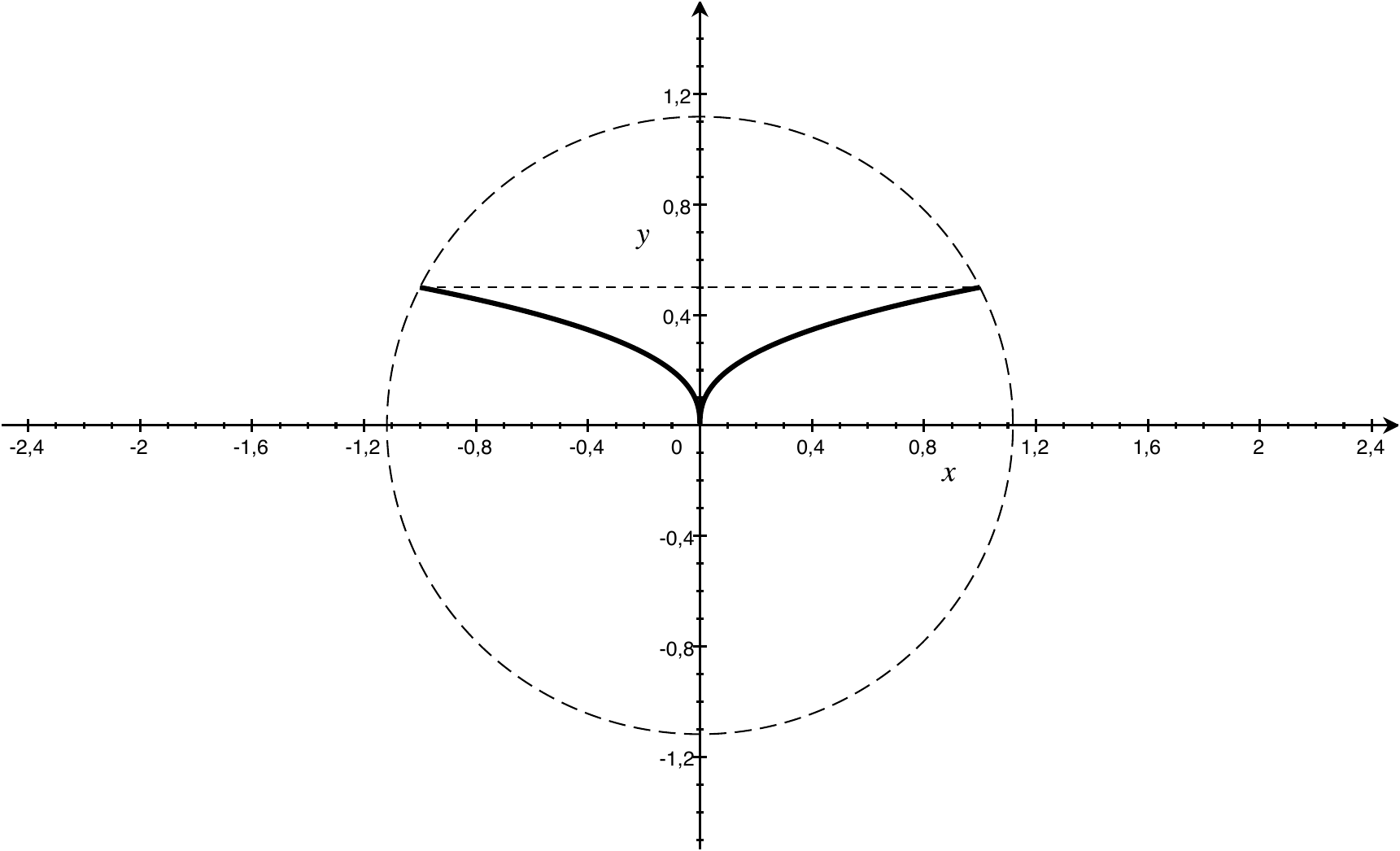}
\caption{In bold line, the profile of a $\beta-$funnel, with $\beta=2/5$, opening $\delta=1/2$ and height $h_0=1/2$. The circle in dashed line has radius $r_0$ given by \eqref{raggiobeta}.}
\end{figure}
Observe that this establishes a bijection between the intervals $(0,h_0]$ and $(0,r_0]$, where we set for simplicity $r_0:=\psi(h_0)$.
Thus, for every $0<r\le r_0$ we have
\[
\mathscr{F}_\beta(\omega;\delta,\psi^{-1}(r))\subseteq B_r(0),\qquad \text{for every}\ \omega\in\mathbb{S}^{N-1}.
\]
In particular, for every $x\in\partial\Omega$ and every $0<r\le r_0$ by definition we have 
\begin{equation}
\label{chuck}
|B_r(x)\setminus\Omega|\ge |(x+\mathscr{F}_\beta(\omega_x;\delta,\psi^{-1}(r)))\setminus\Omega|=|\mathscr{F}_\beta(\mathbf{e}_N;\delta,\psi^{-1}(r))|.
\end{equation}
The last volume can be computed by using Fubini's Theorem 
\begin{equation}
\label{chuck2}|\mathscr{F}_\beta(\mathbf{e}_N;\delta,\psi^{-1}(r))|=\frac{\omega_{N-1}}{\delta^\frac{N-1}{\beta}}\,\int_0^{\psi^{-1}(r)} \tau^\frac{N-1}{\beta}\,d\tau=\frac{\omega_{N-1}}{\delta^\frac{N-1}{\beta}}\,\frac{\beta}{N-1+\beta}\,\Big(\psi^{-1}(r)\Big)^\frac{N-1+\beta}{\beta}.
\end{equation}
In order to give a lower bound for the last term, we observe that 
\[
\psi(h)=\sqrt{h^2+\left(\frac{h}{\delta}\right)^\frac{2}{\beta}}\le \max\left\{1,\left(\frac{h_0^{1-\beta}}{\delta}\right)^{\frac{1}{\beta}}\right\}\,\sqrt{2}\,h,\qquad\text{for}\ 0<h\le h_0,
\]
so that
\[
h\le\psi^{-1}\left(\max\left\{1,\left(\frac{h_0^{1-\beta}}{\delta}\right)^{\frac{1}{\beta}}\right\}\,\sqrt{2}\,h\right),\qquad\text{for every}\ 0<h\le h_0.
\]
By introducing
\[
s=\max\left\{1,\left(\frac{h_0^{1-\beta}}{\delta}\right)^{\frac{1}{\beta}}\right\}\,\sqrt{2}\,h\in \left(0,\max\left\{1,\left(\frac{h_0^{1-\beta}}{\delta}\right)^{\frac{1}{\beta}}\right\}\,\sqrt{2}\,h_0\right],
\]
in particular, we get
\[
\psi^{-1}(s)\ge \frac{s}{\sqrt{2}}\,\frac{1}{\max\left\{1,\left(\dfrac{h_0^{1-\beta}}{\delta}\right)^{\frac{1}{\beta}}\right\}},\qquad \text{for every}\ s \in  \left(0,\max\left\{1,\left(\frac{h_0^{1-\beta}}{\delta}\right)^{\frac{1}{\beta}}\right\}\,\sqrt{2}\,h_0\right].
\]
By observing that $r_0$ belongs to the last interval\footnote{Indeed, from the above discussion we have
\[
r_0=\psi(h_0)\le \max\left\{1,\left(\frac{h_0^{1-\beta}}{\delta}\right)^{\frac{1}{\beta}}\right\}\,\sqrt{2}\,h_0.
\]}, for every $0<r\le r_0$ we thus get from \eqref{chuck} and \eqref{chuck2}
\[
\begin{split}
\left(\frac{r_0}{r}\right)^{(N-1)\,\left(\frac{1}{\beta}-1\right)}\,\dfrac{|B_r(x)\setminus \Omega|}{|B_r(x)|}&\ge c_{N,\beta}\,\left(\sqrt{h_0^2+\left(\frac{h_0}{\delta}\right)^\frac{2}{\beta}}\right)^{(N-1)\,\left(\frac{1}{\beta}-1\right)}\\
&\times\frac{1}{\delta^\frac{N-1}{\beta}}\,\left(\min\left\{1,\left(\dfrac{\delta}{h_0^{1-\beta}}\right)^{\frac{1}{\beta}}\right\}\right)^\frac{N-1+\beta}{\beta},
\end{split}
\]
where 
\[
c_{N,\beta}=\frac{\omega_{N-1}}{\omega_N}\,\frac{\beta}{N-1+\beta}\,\left(\frac{1}{\sqrt{2}}\right)^\frac{N-1+\beta}{\beta}.
\]
We now observe that 
\[
\sqrt{h_0^2+\left(\frac{h_0}{\delta}\right)^\frac{2}{\beta}}=h_0\,\sqrt{1+\left(\frac{h_0^{1-\beta}}{\delta}\right)^\frac{2}{\beta}},
\]
thus from the previous estimate we get
\[
\begin{split}
\theta^*_{\Omega,r_0}(t)&\ge c_{N,\beta}\,\left(\sqrt{1+\left(\frac{h_0^{1-\beta}}{\delta}\right)^\frac{2}{\beta}}\right)^{(N-1)\,\left(\frac{1}{\beta}-1\right)}\,\left(\frac{h_0^{1-\beta}}{\delta}\right)^\frac{N-1}{\beta}\,\left(\min\left\{1,\left(\dfrac{\delta}{h_0^{1-\beta}}\right)^{\frac{1}{\beta}}\right\}\right)^\frac{N-1+\beta}{\beta}.
\end{split}
\]
We now get the claimed lower bound by simply using that $\sqrt{1+\tau^2}\ge |\tau|$.
This concludes the proof.
\end{proof}

\begin{remark}
We observe that for $\beta=1$, the uniform exterior $\beta-$funnel condition boils down to a uniform exterior {\it cone} condition. In this case, in the previous lemma we have $t=0$ and the condition $\theta^*_{\Omega,r_0}(0)>0$ amounts to the classical measure density condition
\[
\inf\left\{\dfrac{|B_r(x)\setminus \Omega|}{|B_r(x)|}\, : x\in \partial \Omega,\ 0 < r \le r_0\right\}>0.
\]
The lower bound in \eqref{lbindex} depends in this case only on the opening of the cone $\delta$, as it is natural.
\end{remark}

\end{document}